\documentclass[12pt]{article}
\usepackage[margin=1in]{geometry}
\usepackage{authblk}
\usepackage{mkolar_definitions,subfigure,graphicx}

\begin{document}

\title{Detecting Anomalous Activity on Networks with\\ the Graph Fourier Scan Statistic\footnote{This research is supported in part by AFOSR under grant FA9550-10-1-0382, NSF under grants DMS-1223137 and IIS-1116458.}}
\author[1]{James Sharpnack}
\author[2]{Alessandro Rinaldo}
\author[2]{Aarti Singh}
\affil[1]{Department of Mathematics, University of California San Diego}
\affil[2]{Department of Statistics, Carnegie Mellon University} 
\affil[3]{Machine Learning Department, Carnegie Mellon University}

\maketitle

\begin{abstract}
We consider the problem of deciding, based on a single noisy measurement at each vertex of a given graph, whether the underlying unknown signal is constant over the graph or there exists a cluster of vertices with anomalous activation. This problem is relevant to several applications such as 
surveillance, disease outbreak detection, biomedical imaging, environmental monitoring, etc. 
Since the activations in these problems often tend to be localized to small groups of vertices in the graphs, we model such activity by a class of signals that are supported over a (possibly disconnected) cluster with low cut size relative to its size. 
We analyze the corresponding generalized likelihood ratio (GLR) statistics and relate it to the problem of finding a sparsest cut in the graph. We develop a tractable relaxation of the GLR statistic based on the combinatorial Laplacian of the graph, which we call the graph Fourier scan statistic, and analyze its properties. We show how its performance as a testing procedure depends directly on the spectrum of the graph, and use this result to explicitly derive its asymptotic properties on a few significant graph topologies. Finally, we demonstrate theoretically and with simulations that the graph Fourier scan statistic can outperform na\"{i}ve testing procedures based on global averaging and vertex-wise thresholding. 
We also demonstrate the usefulness of the GFSS by analyzing groundwater Arsenic concentrations from a U.S.~Geological Survey dataset.
\end{abstract}


\section{Introduction}


In this article, we will take a statistical approach to detecting signals that 
are localized over a graph. 
Signal detection on graphs is relevant in a variety of scientific areas, such as surveillance,  disease outbreak detection, biomedical imaging, detection using a sensor network, 
gene network analysis, environmental monitoring and malware detection over a computer network. 
Recently, the use of graphs to extend traditional methods of signal processing to irregular domains has been proposed \cite{shuman2013signal,sandryhaila2014discrete,
coifman2006diffusion,murtagh2007}.
While this work has largely focused on extending Fourier and wavelet analysis to graphs, little is known about the statistical efficiency of the recently proposed methodology.
We show that the Fourier transform over graphs, defined in \cite{hammond2011wavelets}, can be used to detect anomalous patterns over graphs by constructing the Graph Fourier Scan Statistic (GFSS), a novel statistic based on spectral graph theory.
We demonstrate the connection between the GFSS and the recently proposed Spectral Scan Statistic \cite{sharpnack2012changepoint}, and provide strong theoretical guarantees.

Throughout this work, we will assume that there is a known, fixed, undirected graph with $p$ vertices (denoted by the set $V = \{1,\ldots,p\}$), $m$ edges denoted by pairs $(i,j) \in E \subseteq V \times V$, and \smash{$p \times p$} weighted adjacency matrix \smash{$\Wb$} (where the weight \smash{$W_{i,j} = W_{j,i} \ge 0$} denotes the `strength' of the connection between vertices \smash{$(i,j) \in E$}).
Assume that we observe a single high-dimensional measurement $\yb$ over the graph, whereby for each vertex of the graph, \smash{$i \in V$}, we make a single, Gaussian-distributed observation \smash{$y_i$}.
In the context of sensor networks, the measurements $y_i$ are the values reported by each sensor, and the edge weights reflect beliefs about how similar the measurements of two sensors should be.
The measurements $y_i$ are noisy, and we are interested in determining if there is a region within the network where these observations are abnormally high.
Specifically, we are concerned with the basic but fundamental task of deciding whether there is a `cluster' of vertices within the graph, \smash{$C \subset V$}, such that in expectation the observation, \smash{$\EE[y_i]$}, is larger for \smash{$i \in C$} than for \smash{$i \notin C$}.
In Section II, we will define precisely our statistical framework, including the assumptions placed on the cluster \smash{$C$} and observations \smash{$\yb$} in relation to the graph.
In order to motivate the problem and introduce the GFSS, let us consider the following real data example.

\subsection{Arsenic Ground-water Concentrations in Idaho}

Ground-water contamination remains a serious issue globally, where aging infrastructure, shifting population densities, and climate change are among the contributing factors.
A study published in 1999, reports levels of Arsenic (As) contamination measured in 20,043 wells throughout the United States \cite{focazio1999retrospective}.
In order to illustrate the usefulness of the GFSS, we analyze the As concentration with the purpose of determining if there is a region that has elevated incidence of high As levels.
We will focus on the tested wells within Idaho, which was selected arbitrarily from the other US states.
We construct a graph between the wells, where each vertex is a tested well, by creating an edge between two vertices (wells) if either is the $k$th nearest neighbor of the other (See Figure \ref{fig0}).
For easy visualization, we subsampled the wells by randomly selecting 219 (roughly 10\%) of the 2,191 of the Idaho wells. 
We preprocessed the data by forming the indicator variable, $y_i$, which was $1$ if the measurement made at the $i$th well was greater than 10 ppm and $0$ otherwise (and we will denote the $p$ dimensional vector, $\yb$).
We then applied a standardization which we describe in Section VI.A.
The statistical problem that we address in this paper is testing if there is a well-connected set of wells, $C$, such that the measurements, $\yb$, are abnormally high within the active set $C$.

\begin{figure}[!htbp]
\centering
\mbox{
\subfigure{\includegraphics[width=2.1in]{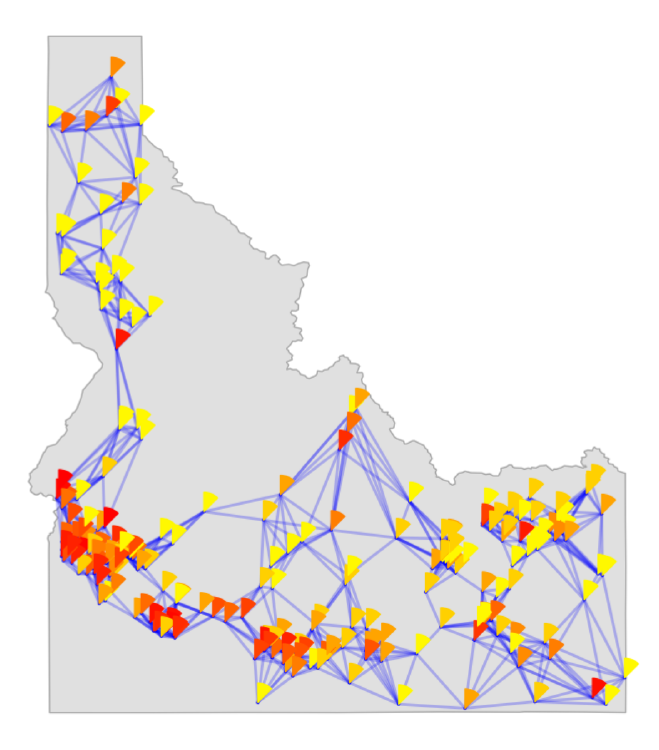}}
\hspace{1cm}
\subfigure{\includegraphics[width=2.3in]{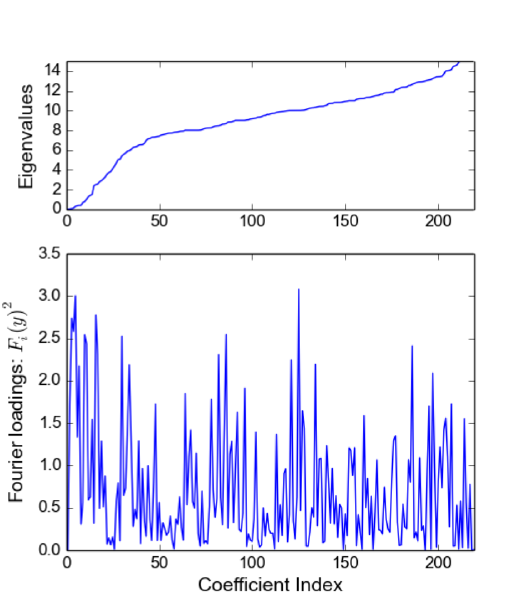}}
}
\caption{\small{\textbf{(Ground-water Arsenic Concentrations)} The As concentrations within Idaho after  (left) where red depicts higher concentrations and yellow depicts lower concentrations.  The ordered eigenvalues of the well network are plotted (top right) and the Fourier loadings \smash{$\{\Fcal_i(\yb)^2 = (\ub_i^\top \yb)^2\}_{i=1}^{219}$} are plotted (bottom right).  The index in the x-axis for the eigenvalues and Fourier loadings match so that the $i$th index corresponds to the pair $\lambda_i, \ub_i$.}}
\label{fig0}
\end{figure}

\subsection{Graph Fourier Scan Statistic}
Traditional statistical methods, such as wavelet denoising (i.e.~Haar and Daubechies wavelets) that employ the standard multi-resolution analysis (see \cite{hardle1998wavelets}) are not adapted to irregular domains and sensor distributions that are not grid-like.
With this in mind, a natural algorithm for the detection of such anomalous clusters of activity is the generalized likelihood ratio test (GLRT) (also known as the scan statistic or matched filter). 
Under a signal plus Gaussian noise model, this procedure entails scanning over all permitted clusters and hence is computationally very intensive. 
In \cite{sharpnack2012changepoint}, the Spectral Scan Statistic (SSS) was proposed as a relaxation of the combinatorial GLRT.
The statistical power of the detector, constructed by thresholding the SSS, was characterized using spectral graph theory.
In this paper, we propose another detector which is a low-pass filter based on the graph Fourier transform.
We will show that the resulting Graph Fourier Scan Statistic (GFSS), is in fact a further relaxation of the SSS, but because of its particular form it allows us to very precisely characterize its statistical power and construct an adaptive counterpart. 

We will begin by introducing a graph Fourier transform, which has been previously proposed in \cite{hammond2011wavelets} (but other transforms have been proposed, as in \cite{sandryhaila2014discrete}).
Through the graph Fourier transform, we will define the GFSS, which we introduce next.
Define the {\em combinatorial Laplacian} matrix \smash{$\Delta = \Db - \Wb$}, where \smash{$\Db = \diag\{ d_i \}_{i = 1}^p$} is the diagonal matrix of vertex degrees, \smash{$d_i = \sum_{j = 1}^p W_{i,j}$}.  
We will denote the eigenvalues and eigenvectors of $\Delta$ with \smash{$\{\lambda_i, \ub_i \}_{i = 1}^p$} respectively, where we order the eigenvalues in increasing order.
Hence, if $\Ub$ is the \smash{$p \times p$} matrix where the $i$th column is the eigenvector $\ub_i$ and $\Lambda = \diag\{ \lambda_i\}_{i=1}^p$ then we have 
\[
\Delta = \Ub \Lambda \Ub^\top.
\]
For the measurement vector $\yb$ over the vertices, the graph Fourier transform is \smash{$\Fcal(\yb) = \Ub^\top \yb$}.
Then the coordinate \smash{$\Fcal_i(\yb) = \ub_i^\top \yb$} for $i$ small are the low frequency components of $\yb$ and for $i$ large are the high frequency components.
In fact, the eigenbasis of the graph Laplacian is commonly used for statistical methods over graphs and point clouds in machine learning.
Much of this work has focused on dimension reduction and clustering \cite{belkin2001laplacian,ng2002spectral,balakrishnan2011noise}, there has been some work on using the Laplacian for regression and testing \cite{nilsson2007regression,sharpnack2010identifying}.
We demonstrate with the GFSS, and its theoretical analysis, another aspect of the Laplacian eigenbasis in a statistical context.

In order to construct the GFSS, consider a low-pass filter, $G$, that passes the low-frequency components of $\yb$ and attenuates (shrinks) the high-frequency components,
\[
G(\yb) = \sum_{i = 2}^p h(\lambda_i) (\ub_i^\top \yb) \ub_i, \quad h(\lambda_i) = \min\Big\{ 1, \sqrt{\frac{\rho}{\lambda_i}} \Big\},
\]
where $\rho > 0$ is a tuning parameter.
Because $\lambda_i$ is increasing in $i$ the attenuation factor, $h(\lambda_i)$, is $1$ for $i$ small enough and is non-increasing in $i$.
Then we {\em define the Graph Fourier Scan Statistic} as the energy of the attenuated signal with an adjustment for the amount of the attenuation (we let $\|.\|$ denote the $\ell_2$ norm),
\begin{equation}
\label{eq:GFSS}
\begin{aligned}
\hat t &= \| G(\yb) \|^2 - \sum_{i = 2}^p h(\lambda_i)^2 \\
&= \sum_{i = 1}^p \min\Big\{ 1, \frac{\rho}{\lambda_i} \Big\} \left[ (\ub_i^\top \yb)^2 - 1 \right].
\end{aligned}
\end{equation}
We will explain why the first eigenvector $\ub_1$ is ignored in Section IV (notice that the index of the sum begins at 2).
We should note here that for any graph Laplacian, $\lambda_1 = 0$ and $u_{1,i} = p^{-1/2}$ for all $i \in V$. 
If the GFSS is abnormally large then a large amount of the signal $\yb$ is in the low-frequency components.
We will see in Section IV.A that this occurs when there is a well-connected cluster $C$ of vertices that have an abnormally large signal.

In Figure \ref{fig0}, we have displayed the eigenvalues in increasing order and the squared graph Fourier coefficients (where the index of the eigenvalues matches the index of the coefficients), $\Fcal_i(\yb)^2$, for the Idaho As concentrations.
Because the linear filter $G(\yb)$ focuses the sensing energy on the low frequency components, the GFSS will be high if the Fourier loadings $\Fcal_i(\yb)^2$ are large for smaller $i$.

By forming a k-nearest neighbor (kNN) graph over all 2,191 wells in Idaho with $k = 8$ and applying the GFSS with $\rho = \lambda_{109}$ (the $109$th smallest eigenvalue, which was selected simply because $109 = \lfloor 0.05 (2191) \rfloor$).
The GFSS statistic evaluates to $697.1$ and we can obtain a P-value $< 10^{-5}$ by a permutation test (explained in Section VI.A).
This indicates that we can be confident that the probability of obtaining a high As measurement is non-constant throughout the graph. 

Recall that we also subsampled the well measurements, to form a kNN graph ($k=8$) over 219 wells (as shown in Figure \ref{fig0}).
By selecting $\rho = \lambda_{10}$ which is selected by the same rule as before ($10 = \lfloor 0.05 (219) \rfloor$), the GFSS also obtains a P-value $< 10^{-5}$.
So, despite the fact that we used $10\%$ of the samples in this dataset, we can still conclude with confidence that the signal is not identically distributed over the graph.
With this knowledge, targeted ground-water treatment could be recommended and further statistical analysis for locating the contamination would be warranted.
After we make a thorough case for the GFSS from a theoretical perspective, we will return to the As detection example in Section VI.A.



\subsection{Related Work}

The problem of statistical hypothesis testing of 
graph-structured activation has received some attention recently. 
The GLRT for graphs, also known as the graph scan statistic, is discussed in \cite{priebe_graphscan}. 
Theoretical properties of the GLRT for some specific topologies and 
specific signal classes have also been derived, e.g.
detecting an interval in a line graph or geometric
shapes such as rectangles, disks or ellipses in a lattice graph
\cite{castro:05}, path of activation in a tree or lattice \cite{arias2008searching}, 
or nonparametric
shapes in a lattice graph \cite{arias2011detection}. In these settings, scanning over
the entire signal class or over an epsilon-net for the signal 
class is often feasible and has been shown to have near-optimal
statistical performance. However, for general graphs and signal classes 
these detectors are infeasible, either because the scan
involves too many patterns or due to lack of constructive
ways to obtain an epsilon-net. While there has been some
work on developing fast graph subset scanning methods \cite{daniel_graphscan},
these greedy methods sacrifice statistical power. Also, there is 
work on developing Fourier basis and wavelets for graphs (c.f. \cite{shuman2013signal} and references therein),
which can potentially serve as an epsilon-net, however the 
approximation properties of such basis are not well characterized. 
An exception is \cite{sharpnack2012detecting} where graph wavelets were constructed using a 
spanning tree and statistical properties of the corresponding wavelet 
detector have been characterized. 
In \cite{addario2010combinatorial}, 
the authors consider the
complete graph and study detection under some combinatorial
classes such as signals supported over cliques, bi-cliques, and spanning trees. 
They establish lower bounds on the
performance of any detector, and provide upper bounds for
some simple but sub-optimal detectors such as averaging all
node observations and thresholding. 


We build on our previous findings in \cite{sharpnack2012changepoint} where the Spectral Scan Statistic was proposed as a convex spectral relaxation of the GLRT and characterize its statistical performance. 
In another recent work \cite{sharpnack2013submod}, we have also developed a different convex relaxation of the GLRT using Lovasz extension and characterized its properties for detecting graph-structured signals. 
A comparison of our prior work \cite{sharpnack2012detecting,sharpnack2012changepoint,sharpnack2013submod} appears in \cite{sharpnackGlobalSIP13}.  
Despite the empirical success of the SSS in \cite{sharpnack2012changepoint}, the statistical guarantees made are in some cases dominated by the guarantees obtained for the energy test statistic (to be introduces in Section III.B) which does not take the graph structure into account.
The GFSS attains superior theoretical performance which always outperforms the energy statistic (except in cases in which the graph structure is misleading).
Moreover, because the GFSS is formed by attenuating high frequency components via the graph Fourier transform (as in \cite{hammond2011wavelets}), this paper provides a statistical justification for the use of the combinatorial Laplacian to derive a graph Fourier analysis.
Furthermore, the SSS requires perfect knowledge of the tuning parameter $\rho$, which is not known in general.
To this end, we form the adaptive GFSS, which automatically selects $\rho$.
In practice, the adaptive GFSS significantly outperforms the GFSS with a heuristic choice of $\rho$.
The GFSS also may be preferable to more complicated procedures because it is based on a linear filter of the measurements $\yb$, which in some computational settings may be advantageous.

\subsection{Contributions}

Our contributions are as follows. 
(1) We examine a new alternative hypothesis, which we call the graph-structured $H_1$, which generalizes the piece-wise constant graph-structured $H_1$ proposed in \cite{sharpnack2012changepoint}. 
(2) Following the derivation of the SSS in \cite{sharpnack2012changepoint}, we show the relationship between the GFSS, SSS, and GLRT.  
(3) In our main theoretical result, we show that the performance of the GFSS depends explicitly on the spectral properties of the graph. 
(4) Because the GFSS requires the specification of the tuning parameter, $\rho$, we develop an adaptive version of the GFSS that automatically selects $\rho$.
We extend our theory to this test.
(5) Using such results we are able to characterize in a very explicit form the performance of the GFSS on a few notable graph topologies and demonstrate its superiority over detectors that do not take into account the graph structure.
(6) We demonstrate the usefulness of the GFSS with the partially simulated Arsenic concentration dataset.

\section{Problem Setup}

Detection involves the fundamental statistical question: are we observing merely noise or is there some signal amidst this noise?
While the As contamination example in Section I.A involves binary measurements, for ease of presentation, we will work with Gaussian measurements with the understanding that many of the results derived may be extended easily to binary observations via subGaussian theory.
We begin by outlining the basic problem of detecting a signal in Gaussian noise, then we will dive into graph-structured signals and the corresponding detection problem.
First we begin with the Gaussian sequence space model, in which we make {\bf one} observation at each node of the graph, yielding a vector $\yb \in\RR^p$ which is modeled as
\begin{equation}\label{eq:model}
\yb = \xb + \epsilonb,
\end{equation}
where $\xb \in \RR^p$ is the unknown signal and $\epsilonb \sim N(0,\sigma^2 \Ib_p)$ is Gaussian noise, with $\sigma^2$ known. 
While the basic detection problem is to determine whether the signal $\xb$ is constant or not, we analyze three different alternative hypotheses: unstructured signal, piece-wise constant graph-structured signal, and the more general graph-structured signal with differential activation.

\emph{Unstructured $H_1$.} In order to make clear the improvements that the graph-structure affords us, we will consider the testing problem without any such structure.  
Throughout this paper we will also let $\one = (1,\ldots,1)$ denote the all $1$s vector, and $\bar \xb = n^{-1} (\sum_{i=1}^p x_i) \one$.
We define the class of unstructured signals, parametrized by a signal strength parameter $\mu > 0$, as $\Xcal_U(\mu) = \{ \xb \in \RR^p : \| \xb - \bar \xb \| \ge \mu \}$, the complement of the open ball of radius $\mu$ in the subspace orthogonal to $\one$.
We consider only the subspace orthogonal to $\one$ because the null is invariant under changes of $\xb$ within this space.
Then the basic `unstructured' hypothesis testing problem is,
\[
H_0: \xb = \bar \xb \textrm{\quad v.s. \quad} H_1^U : \xb \in \Xcal_U(\mu)
\]
Because there is no a priori structure, in this setting we are compelled to use tests that are invariant under arbitrary permutation of the coordinates of $\yb$ and $\xb$.
We will now outline the graph-structured alternative hypotheses. 

\emph{Piece-wise constant graph-structured $H_1$.} 
Following \cite{sharpnack2012changepoint}, will assume that there are two regions of constant signal for $\xb$, namely that there exists a (possibly disconnected) subset $C \subset V$ ($C \notin \{\emptyset, V\}$) such that $\xb$ is constant within both $C$ and its complement $\bar C = V \backslash C$.
We consider the class of signals that are piece-wise constant on $C$ and $\bar C$, 
and parametrized by a signal strength parameter $\mu$, defined as 
\[
\Xcal_{PC}(\mu,C) = \{ \xb = \alpha \one + \delta \one_C : \alpha, \delta \in \RR \} \cap \Xcal_U(\mu)
\]
where $\one_C \in \RR^p$ such that $\one_{C,i} = 1$ if $i \in C$ and $0$ otherwise.  
The parameter $\alpha$ can be thought of as the magnitude of the background signal and is a nuisance parameter, while $\delta$ quantifies the gap in signal between the two clusters.
For the signal $\xb = \alpha \one + \delta \one_C$ to be contained in $\Xcal_U(\mu)$, it is required that $\| \xb - \bar \xb\| = \sqrt{\frac{|C||\bar C|}{p}} |\delta| \ge \mu$.

We will not assume any knowledge of the true cluster $C$, other than that it belongs to a given class $\Ccal \subset 2^V$ that we define next. This class provides a good model for activations that are localized on 
the graph as we will see. 
Formally, we define, for some $\rho > 0$ (which is the same $\rho$ in the definition of the GFSS),
\begin{equation}
\label{eqn:Cclass}
\Ccal = \Ccal(\rho) =  \left\{C \subset V, C \neq \emptyset \colon \frac{ \Wb(\partial C) }{ |C||\bar C|} \le \frac {\rho} {p} \right\},
\end{equation}
where $\partial C = \{(i,j) \in E: i \in C, j \in \bar C \}$ is the boundary of $C$ and $\Wb(\partial C) = \sum_{(i,j) \in \partial C} W_{i,j}$. 
Note that $\Ccal$ is a symmetric class in the sense that $C \in \Ccal$ if and only if $\bar C \in \Ccal$. 
The quantity $\frac{p \Wb(\partial C) }{ |C||\bar C|}$ is known in the graph theory literature as the {\bf cut sparsity} \cite{vazirani2001approximation} and is equivalent, up to factor of $2$, to the {\bf cut expansion} ($\frac{\Wb(\partial C)}{\min\{|C|,|\bar C|\}}$):
\[
\frac{\Wb(\partial C)}{\min\{|C|,|\bar C|\}} \le \frac{p \Wb(\partial C) }{ |C||\bar C|} \le 2 \frac{\Wb(\partial C)}{\min\{|C|,|\bar C|\}}
\]
The cut expansion of a vertex set $C$ is a measure of the size of the boundary relative to the size of $C$. Notice that for the same size of activation, a signal that is localized to a group of well connected nodes on the graph has a smaller cut sparsity and cut expansion than a signal which is distributed over isolated nodes. 
Thus, the class provides a good model for signals that are localized over the graph. 
Note that this definition is much less restrictive than existing work, e.g. \cite{castro:05} considers intervals, rectangles, ellipses, and similar geometrical shapes, or \cite{addario2010combinatorial} considers cliques, stars, spanning trees.
The only other work which considers general non-parametric shapes is \cite{arias2011detection}, 
however it only considers lattice graphs and it is not clear how to extend the signal class 
definition used in that work to general graphs.

Define the class of all piece-wise constant signals with signal strength $\mu$ and cut sparsity bounded by $\rho$ as $\Xcal_{PC}(\mu,\rho) = \cup_{C \in \Ccal(\rho)} \Xcal_{PC}(\mu,C)$.
Then we will consider the hypothesis testing problem,
\[
H_0: \xb = \bar \xb \textrm{\quad v.s. \quad} H_1^{PC} : \xb \in \Xcal_{PC}(\mu,\rho).
\]
Because the piece-wise constant assumption may be restrictive, we will endeavor to generalize it to a larger alternative space.

\emph{Graph-structured $H_1$.} We now consider a more general form of alternative, in which the signal is graph-structured, but not necessarily constant over clusters of activation.
Specifically, we will assume that there is again a true cluster $C \in \Ccal(\rho)$ within which the signal differs little and across which the signal differs highly.
Formally, we define the class of graph-structured signals as 
\[
\Xcal_{S}(\mu,\rho) = \left\{ \xb \in \RR^p : \left| \frac{\one_C^\top \xb}{|C|} - \frac{\one_{\bar C}^\top \xb}{|\bar C|} \right| \sqrt{\frac{|C||\bar C|}{p}} \ge \mu, C \in \Ccal \right\}.
\]
Notice that if $\xb \in \Xcal_S(\mu,\rho)$ then $\| \xb - \bar \xb \| \ge \mu$, so $\Xcal_S(\mu,\rho) \subset \Xcal_U(\mu)$.
Furthermore, if $\xb = \alpha \one + \delta \one_C \in \Xcal_{PC} (\mu, \rho)$ then
\[
\left| \frac{\one_C^\top \xb}{|C|} - \frac{\one_{\bar C}^\top \xb}{|\bar C|} \right| = |\delta|.
\]
Hence, $\Xcal_{PC}(\mu,\rho) \subset \Xcal_S(\mu,\rho)$.
This induces the following hypothesis testing problem,
\[
H_0: \xb = \bar \xb \textrm{\quad v.s. \quad} H_1^S : \xb \in \Xcal_S(\mu,\rho).
\]
Whenever possible we will make statements about this non-constant alternative, for the sake of generality.
We outline in Section VI.E how signals in this class may arise by subsampling vertices of the cluster $C$.

\subsection{Distinguishability of $H_0$ and $H_1$}

We will analyze asymptotic conditions under which the hypothesis testing problems described above are statistically feasible, in a sense made precise in the next definition.
We will assume that the size of the graph $p$ increases and the relevant parameters of the model, $\mu$, $\sigma$, $\rho$, and eigen-spectrum of $\Delta$, change with $p$ as well, even though we will not make such dependence explicit in our notation for ease of readability.  
Our results establish conditions for asymptotic disinguishability as a function of the SNR $\mu/\sigma$, $\rho$, and the spectrum of the graph.

\begin{definition}
For a given statistic $s({\bf y})$ and threshold $\tau \in \RR$, let $T = T({\bf y})$ be $1$ if $s({\bf y}) > \tau$ and $0$ otherwise. 
Recall that $H_0$ and $H_1$ index sets of probability measures by which $\yb$ may be distributed.
We say that the hypotheses $H_0$ and $H_1$ are {\bf asymptotically distinguished by the test} $T$ if
\begin{equation}
    \label{eqn:asymp_dist}
  \sup_{\PP_0 \in H_0} \PP_0 \{ T=1 \} \rightarrow 0 \quad \textrm{ and } \quad \sup_{\PP_1 \in H_1} \PP_1 \{ T=0 \} \rightarrow 0, 
  \end{equation}
  where the limit is taken as $p \rightarrow \infty$.
We say that $H_0$ and $H_1$ are {\bf asymptotically indistinguishable} if there does not exist any test for which the above limits hold.
Furthermore, we say that a sequence $\{r_p\}_{p=1}^\infty$ is a {\bf critical SNR}, if for $\mu / \sigma = o(r_p)$, $H_0$ and $H_1$ are asymptotically indistinguishable and for $\mu / \sigma = \omega(r_p)$, $H_0$ and $H_1$ are asymptotically distinguishable. (We denote this with $\mu / \sigma \asymp r_p$.)
\end{definition}


In Section III.A we will produce a lower bound on the critical SNR, and in Section IV we will derive conditions under which the GFSS asymptotically distinguishes $H_0$ from $H_1^{PC}$ and $H_1^S$.
We will say that a test is {\bf adaptive} if it can be performed without knowledge of $\rho$.
Naturally, requiring adaptivity may inhibit the quality of our test, as it has in the detection within Sobolov-type ellipsoids \cite{spokoiny1996adaptive,ji2012sharp}.
We will modify the GFSS to make it adaptive and prove theory regarding its performance.

\section{A Lower Bound and Classical Results}

Our ultimate goal is to give a theory of activity detection on general graphs.
This means that our theorems should apply to all graph structures with only minor, simplifying assumptions, such as connectedness.
But as a validation of the theory and methods that we propose, we will pay particular attention to the implications of our results on specific graph structures.
We begin by introducing the torus graph structure that will serve as a running example for illustrations. 
The reader should in no way interpret this to mean that our results necessarily depend on the idiosyncrasies of the torus graph, such as edge transitivity.

\begin{example}{(Torus Graph)} A torus graph is a lattice or two-dimensional grid that is 
wrapped around so that rightmost vertices are same as leftmost vertices, and topmost 
vertices are same as bottom vertices. Formally, 
the $\ell \times \ell$ torus graph ($p = \ell^2$) is defined as follows. Let the vertex set $V = (\ZZ \mod \ell)^2$ where points $(i_1,i_2),(j_1,j_2)$ are connected by an edge if and only if $|(i_1 - j_1) \mod \ell| + |(i_2 - j_2) \mod \ell| = 1$ (here $|i \mod \ell|$ means the smallest absolute value of representatives).
The class of clusters in the torus under consideration $\Ccal(\rho)$ are those that have sparsity $p \Wb(\partial C) / (|C||\bar C|) \le \rho$.
For example, rectangles of size $k \times k$ within the torus have cut sparsity $4kp / (k^2 (p-k^2)) \asymp 4/k$.
This means that if we would like to include rectangles of size roughly $k \times k$, it is sufficient that $\rho \asymp 4/k$. 
\end{example}

In order to understand the fundamental limitations of the activity detection problem, we review lower bounds on the performance of any testing procedure.
After this information theoretic result, we will review classical theory about the detection of non-zero means under no graph constraints.

\subsection{Lower Bound}

The following lower bound on the critical SNR was derived in \cite{sharpnack2012changepoint}.
The first part is a simple bound on the performance of the oracle (who has knowledge of the active cluster, $C$) based on the Neyman-Pearson lemma.
The second part is more sophisticated and requires that the graph has symmetries that we can exploit, but it will not be satisfied by many graphs.
We will later show that these conditions are satisfied by the specific graph structures that we will analyze in Section VI.

\begin{theorem}{\cite{sharpnack2012changepoint}}
\label{thm:lower_bd}
(a) $H_0$ and $H_1^{PC}, H_1^S$ are asymptotically indistinguishable if $\mu / \sigma = o(1)$. \newline
(b) Suppose that there is a subset of clusters $\Ccal' \subseteq 2^V$ such that all the elements of $\Ccal'$ are disjoint, of the same size ($|C| = c$ for all $C \in \Ccal'$), and
\[
\forall C \in \Ccal', \quad \frac{p \Wb(\partial C)}{|C||\bar C|} \le \frac {\rho}{2}
\]
i.e., elements of $\Ccal'$ belong to the alternative hypothesis with $\rho / 2$ cut sparsity.
Furthermore assume that $\frac{c|\Ccal'|}{p} \rightarrow 1$.
$H_0$ and $H_1^{PC}, H_1^S$ are asymptotically indistinguishable if
\[
\frac \mu \sigma = o(|\Ccal'|^{1/4})
\]
\end{theorem}

We illustrate the usefulness of the lower bound with the following example.
\begin{example}{(Lower Bound for Torus)}
We will construct $\Ccal'$ in Theorem \ref{thm:lower_bd} (b) from disjoint squares of size a constant multiple of $p^{1 - \beta}$, making 
$|\Ccal'|\asymp p^\beta$.
Thus, the critical SNR for $H_0$ versus any of $H_1^{PC}, H_1^S$ for any estimator is greater than $p^{\beta / 4}$. 
\end{example}

In \cite{arias2011detection}, the authors also derive a lower bound which scales as $\sqrt{\log (p/|C|)}$ for detection of
patterns on the lattice graph which include squares of size $|C|$. However, their results 
only hold for clusters consisting of a single connected component, whereas $H_1^{PC}$ allows for 
multiple connected components. Thus, our results indicate that detecting clusters 
with multiple 
connected components is harder than detecting 
a cluster with a single connected component, unless the connected component is really large
i.e. the activation size $|C|$ is of the same order as the graph size $p$. In the latter case, 
both the bound of \cite{arias2011detection} and our result imply an SNR of 
$o(1)$ is insufficient for detection on the torus graph. 

The scaling with the $1/4$th power in our results is not a coincidence.
We will see that the classical results for the unconstrained alternative hypotheses also provide a critical SNR of this form.

\subsection{Classical Results}

In order to understand the inherent difficulty of distinguishing $H_0$ from the unstructured alternative, $H_1^U$, we will recount a result from \cite{ingster2003nonparametric}.

\begin{theorem}
The critical SNR for any test distinguishing $H_0$ from $H_1^U$ is given by,
\[
\frac \mu \sigma \asymp p^{1/4}
\]
and it is achieved by the energy test statistic $\| \yb - \bar \yb \|_2^2$, when it is thresholded at quantiles of the $\chi^2_{p-1}$ distribution.
\end{theorem}

This result highlights the aforementioned $1/4$th power scaling in critical SNRs.
One would hope that in the graph-structured setting, we can tolerate an SNR smaller than this.
We will see that this is achieved by the GFSS, and by its adaptive version in most cases.
For completeness, we will also look at two other test statistics.
Let the statistics $\max_{i \in [p]} |y_i - \bar y|$ and $\one^\top \yb$ be called the {\em max} statistic and the {\em aggregate} statistic respectively.
Then they have the following required SNR's for the piece-wise constant alternative structure, $H_1^{PC}$.

\begin{theorem}
\label{thm:max_agg}
(a) Consider a sequence of draws from $H_1^{C}: \yb = \xb + \epsilon$ with $\xb \in \Xcal_{PC}(\mu,C)$, for the active cluster $C$.  The critical SNR for $H_0$ versus $H_1^{C}$ of the max statistic is between the following
\[
\frac \mu \sigma = \omega(\sqrt{|C|}), \quad  \frac \mu \sigma = o(\sqrt{|C| \log p})
\]
while the upper bound is an equality ($\asymp$) if $\log |C| = o( \log p )$.\newline
(b) Suppose further that the alternative $H_1^{C}$ has the more specific form:
\[
\xb_C = \sqrt{\frac{p}{|C||\bar C|}} \mu \one_C
\]
The critical SNR for $\xb = \zero$ versus $\xb = \xb_C$ of the aggregate test statistic is $\mu / \sigma \asymp \sqrt{|\bar C|/|C|}$.
\end{theorem}

\begin{proof}
(a) follows directly from \cite{ingster2003nonparametric} (Corollary 3.10).
(b) This follows from the fact that under $H_0$, the test statistic is $\Ncal(0,p\sigma^2)$, while under $H_1: \xb = \xb_C$ it has mean $\sqrt{(p |C|)/|\bar C|} \mu$.   
\end{proof}

\begin{remark}
Notice that the critical SNR of the max statistic, Theorem \ref{thm:max_agg} (a), can be significantly worse than the energy and aggregate statistics if $|C| \ge \sqrt p$.
Otherwise the max statistic is superior.
Similarly, (b) provides worse performance than the energy and max statistic if $|C| \le \sqrt p$.
These only hold for the piecewise constant graph structure of $H_1^{PC}$.
We state the results with the dependence on $|C|$, as opposed to their worst case in the class $\Ccal(\rho)$, because one could apply an omnibus test that adapts to whichever test performs better.
\end{remark}


\section{Graph Fourier Scan Statistic}

In order to derive the Graph Fourier Scan Statistic (GFSS), we will consider specifically the piece-wise constant graph structure, $H^{PC}_1$.
Before we arrive at the GFSS, we recall the definition of the Spectral Scan Statistic (SSS).
While the GFSS is shown to be a relaxation of the SSS, we favor the GFSS because it is simple to implement, performs as well as the SSS in practice, and is the basis of the construction of the adaptive GFSS.

\subsection{Derivation of GFSS}

The hypothesis testing problem with signal in $\Xcal_{PC}(\mu,\rho)$ presents
two challenges: (1) the model contains an unbounded nuisance parameter $\alpha
\in \mathbb{R}$ and (2) the alternative hypothesis is comprised of a finite
disjoint union of composite hypotheses indexed by $\mathcal{C}$. These features
set our problem apart from existing work of structured normal means problems
(see, e.g. \cite{castro:05, arias2008searching,
arias2011detection,addario2010combinatorial}), which does not consider nuisance
parameters and relies on a simplified framework consisting of a simple null
hypothesis and a composite hypothesis consisting of disjoint unions of simple
alternatives. 


To derive the GFSS we will first consider the simpler problem of testing the null hypothesis that $\xb = \bar \xb$,
i.e. that the signal is constant, versus the alternative composite hypothesis that
\[
    \xb = \alpha \one + \delta \one_C : \alpha, \delta \in
    \RR, \delta \neq 0, 
\]
for one given non-empty $C \subset V$. A standard approach to solve this
testing problem is to compute the likelihood ratio (LR) statistic
\begin{equation}
\label{eq:LR}
2 \log \Lambda_C(\yb) = \frac{1}{\sigma^2} \frac{p}{|C| |\bar{C}|} \Big( \sum_{v \in C} \tilde{y}_v \Big)^2,
\end{equation}
where $\tilde \yb = \yb - \bar \yb = (\tilde y_v, v \in V)$, 
and to reject the null hypothesis for large values of $ \Lambda_C(\yb)$
(the exact threshold for rejection will depend on the choice of the test significance level). Equation \eqref{eq:LR} was first obtained in
\cite{sharpnack2012changepoint}. In Appendix B, we provide an alternative derivation 
that shows rigorously how we can eliminate the interference caused by the
nuisance parameter by considering test procedures that are independent of
$\alpha$ (or equivalently $\bar \xb$). The formal justification for this choice
is based on the theory of optimal invariant hypothesis testing (see, e.g.,
\cite{lehmann2005testing}) and of uniformly best constant power tests (see
\cite{fouladirad2008optimal,fouladirad2005optimal,fillatre2012,scharf94,baygun.hero:95,wald:43}). 


%

When testing against the more complex composite alternative $\xb \in \{
    \Xcal_{PC}(\mu,C), C \in \mathcal{C}(\rho)\}$, for a given $\rho>0$, 
 it is customary to consider instead the generalized likelihood ratio (GLR) statistic, which in our case reduces to
\[
\hat g =  \max_{C \in \mathcal{C}(\rho)} 2 \sigma^2 \log \Lambda_C(\yb).
\]
With simple algebraic manipulations of the LR statistic \eqref{eq:LR}, we find
that the GLR statistic has a very convenient form which is tied to the spectral
properties of the graph $G$ via its Laplacian. We state it as a result and omit
the simple proof.
\begin{lemma}
\label{lem:GLRTform}
  Let $\Kb = \Ib - \frac{1}{p} \one \one^\top$ and set $\tilde{\yb} = \Kb \yb$. Then
\begin{equation}
\label{eqn:GLRT}
\hat g = \max_{\xb \in \{0, 1\}^p} \frac{\xb^\top \tilde{\yb}\tilde{\yb}^\top \xb}{\xb^\top \Kb \xb} \textrm{ s.t. } \frac{ \xb^\top \Delta \xb}{\xb^\top \Kb \xb} \le \rho,
\end{equation}
where $\Delta$ is the combinatorial Laplacian of the graph $G$.
\end{lemma}

An interesting feature of the GLR statistic is that the program \eqref{eqn:GLRT}
is directly related to the renowned sparsest cut problem in combinatorial
optimization. See Section 2 of \cite{sharpnack2012changepoint} for details.
In order to obtain a tractable relaxation of the GLR statistic
\eqref{eqn:GLRT}, \cite{sharpnack2012changepoint} introduced the Spectral Scan
Statistic (SSS), defined as
\[
\hat{s} = \sup_{\xb \in \RR^p} (\xb^\top \tilde\yb)^2  \textrm{ s.t. }\xb^\top \Delta \xb \le \rho, \| \xb \| \le 1, \xb^\top \one = 0.
\]
Indeed, \cite{sharpnack2012changepoint} proved that the SSS is an upper bound to
the GLRT statistic:
\begin{proposition}
\label{prop:SSS_def}
The GLR statistic is bounded by the SSS: $\hat g \le \hat s$, almost everywhere.
\end{proposition}

Notice that because the domain $\Xcal = \{ \xb \in \RR^n : \xb^\top \Delta \xb \le \rho, \| \xb \| \le 1, \xb^\top \one = 0 \}$ is symmetric around the origin, this is precisely the square of the solution to
\begin{equation}
\label{eqn:sGP}
\sqrt{\hat{s}} = \sup_{\xb \in \RR^n} \xb^\top \yb  \textrm{ s.t. }\xb^\top \Delta \xb \le \rho, \| \xb \| \le 1, \xb^\top \one = 0,
\end{equation}
where we have used the fact that $\xb^\top \tilde\yb = ( (\Ib - \frac{1}{n}\one \one^\top) \xb )^\top \yb = \xb^\top \yb$ because $\xb^\top \one = 0$ within $\Xcal$. 

\begin{proposition}
\label{prop:GFSS_def}
Recall the definition of the GFSS, $\hat t$ in \eqref{eq:GFSS}.
The SSS as a function of $\rho$ can be bounded above and below in the following:
\[
\hat t + \sum_{i = 2}^p \min \left\{1 , \frac{\rho}{\lambda_i} \right\} \le \hat s \le 2 \left(\hat t + \sum_{i = 2}^p \min \left\{1 , \frac{\rho}{\lambda_i} \right\}\right).
\]
\end{proposition}

The proof is provided in Appendix A.
The implication of Propositions \ref{prop:SSS_def} and \ref{prop:GFSS_def} is that the GFSS, $\hat t$, is a relaxation of the GLRT, $\hat g$.
It is not clear, even if it is possible to obtain a poly-time algorithm that can distinguish $H_0$ from $H_1^{PC}$ over {\em any} graph under the critical SNR regime.
The GFSS is a computationally tractable alternative to the GLRT, and as we will see, it is often a vast improvement over the naive test statistics.
Let us consider the GFSS and show what it does in our torus example.

\begin{example}{(GFSS for the Torus)}
It has been shown that the Laplacian eigenvalues of the torus graph are $2 (2 - \cos (2 \pi i_1/\ell) - \cos (2 \pi i_2 / \ell))$ for all $i_1,i_2 \in [\ell]$  (see \cite{sharpnack2010identifying} for a derivation).
The eigenvectors correspond to that of the discrete Fourier transform.
So the GFSS for the Torus graph corresponds to the energy of linear shrinkage in the frequency domain.
\end{example}


\subsection{Theoretical Analysis of GFSS}

A thorough theoretical analysis of the GFSS has several uses.
In Corollary \ref{cor:crit_SNR}, we characterize the critical signal-to-noise ratio, enabling us to determine the strength of the GFSS as a detector on theoretical grounds.
Theorem \ref{thm:main} will be used to form an adaptive version of the GFSS, which will in turn alleviate the need for specifying $\rho$.

The following main result bounds the test statistic under $H_0$ and under the piece-wise constant ($H_1^{PC}$) and the general graph structured ($H_1^S$).
It is based on the concentration of weighted sums of independent $\chi^2$ random variables found in \cite{laurent2000adaptive}.

\begin{theorem}
\label{thm:main}
Under the null hypothesis $H_0$, with probability at least $1 - \alpha$ where $\alpha \in (0,1)$,
\begin{equation}
\label{eq:null_control}
\hat t \le 2 \left( \sqrt{\sum_{i=2}^p \min \big\{1 , \frac{\rho^2}{\lambda_i^2} \big\} \log (1 / \alpha) } + \log (1 / \alpha) \right).
\end{equation}
Under the alternative hypotheses, $H^{PC}_1, H^S_1$, with probability at least $1 - \gamma$ where $\gamma \in (0,1)$,
\begin{equation}
\label{eq:alt_control}
\hat t \ge \frac{\mu^2}{2 \sigma^2} - \frac {2\mu} \sigma \sqrt{\log \frac 2 \gamma} - 2 \sqrt{\sum_{i=2}^p \min \left\{1 , \frac{\rho^2}{\lambda_i^2} \right\} \log \frac 2\gamma},
\end{equation}
for $\mu/\sigma$ large enough.
\end{theorem}

The proof is provided in Appendix A.
Theorem \ref{thm:main} shows that by setting a threshold to be the right hand side of \eqref{eq:null_control}, we have a level $\alpha$ test.
If we then set the right hand side of \eqref{eq:alt_control} to be this threshold, and solve for $\mu/\sigma$, then we get the lowest SNR such that the test has power $1 - \gamma$ under the alternative.
The result below allows us to compare the GFSS to other tests on asymptotic theoretical grounds.

\begin{corollary}
\label{cor:crit_SNR}
The GFSS $\hat t$ can asymptotically distinguish $H_0$ from $H_1^{PC}, H_1^S$ if the SNR is stronger than
\[
\frac{\mu}{\sigma} = \omega \left( \sum_{i = 2}^p \min \left\{1 , \frac{\rho^2}{\lambda_i^2} \right\} \right)^{1/4}.
\]
\end{corollary}

Most notably the critical SNR is lower than $p^{1/4}$ which is the critical SNR enjoyed by the energy test statistic.
Comparing this to the analogous results for the SSS in \cite{sharpnack2012changepoint} (Cor.~8), we see that this is a significant improvement.
The most unreasonable assumption that we have made thus far is that the cut sparsity, $\rho$, is known.
This unreasonable advantage is especially apparent when we compare the GFSS to the max and aggregate statistics that do not require the knowledge of $\rho$.
The following section develops a test that adapts to $\rho$.

Before we delve into that, a remark on the computational complexity of the proposed
test is in order. The worst-case runtime of computing all the eigenvalues of a $p\times p$ 
matrix is cubic in $p$. Thus, we have clearly demonstrated a test that is 
computationally feasible for any graph topology i.e. it runs in polynomial (cubic) time 
in the size of the graph $p$, compared to the GLRT whose computation can be exponential
in $p$ in the worst case. However, cubic computational complexity might be prohibitive 
for very large graphs. 
It turns out that the GFSS can be calculated with just the top $j$ eigenvectors and a Laplacian solver.
Specifically, if $j = \max\{ i : \lambda_i < \rho \}$ and let $\Pb_j$ be the projection onto the span of $\{\ub_i\}_{i=2}^j$ then the GFSS can be written as 
\[
\yb^\top \Pb_j \yb + \rho \yb^\top (\Ib - \Pb_j) \Delta^{\dagger} (\Ib - \Pb_j) \yb - \sum_{i = 2}^p \min\big\{ 1 , \frac{\rho}{\lambda_i} \big\}.
\]
The first term requires the computation of the top $j$ eigenspace, while the second term requires a Laplacian solver.
One can observe that the final term is the expected value of the first two terms applied to a vector drawn from the p-dimensional standard normal, indicating that it can be approximated by Monte carlo sampling.
The computation of the first two terms take time $O(k p^2 + m \textrm{polylog}(m))$ by using the fast Laplacian solvers of \cite{koutis2010approaching}.
Furthermore, when the GFSS will be used on multiple measurement vectors (such as multiple measurements in time), then $\Pb_j$ needs to only be computed once. 



\section{The Adaptive GFSS}

Notice that the SSS and GFSS require that we prespecify the cut sparsity parameter, $\rho$.
While the user may have certain shapes in mind, such as large rectangles in a lattice, it is not reasonable to assume that this can be done for arbitrary graph structure.
In order to adapt to $\rho$ we will consider the test statistic, $\hat t(\rho)$, as a function of $\rho$, as it is allowed to vary.

\begin{definition}
Let $\alpha > 0$ and 
\[
\tau(\rho) = 2 \left( \sqrt{\sum_{i=2}^p \min \big\{1 , \frac{\rho^2}{\lambda_i^2} \big\} \log \left( \frac{p-1}{\alpha} \right) } + \log \left( \frac{p-1}{\alpha} \right) \right).
\]
The adaptive GFSS test is the test that rejects $H_0$ if $\exists \rho > 0$ such that
\begin{equation}
\label{eq:tau_function}
\hat t (\rho) > \tau(\rho).
\end{equation}
\end{definition}
 
As we will now show, in order to compute the entire curve $\hat t(\rho)$, it is sufficient to evaluate $\hat t(\rho)$ only at $p-1$ points.
This is because $\hat t(\rho)$ is piecewise linear with knots at the eigenvalues.
Also, $\tau(\rho)$ is similarly well behaved.
Let $j = \max\{i : \lambda_i \le \rho \}$ then
\[
\hat t(\rho) = \rho \sum_{i = j+1}^p \frac{(\ub_i^\top \yb)^2 - 1}{\lambda_i} + \sum_{i=2}^j ((\ub_i^\top \yb)^2 - 1).
\]
Hence, $\hat t(\rho)$ is piecewise linear with knots at $\{\lambda_i\}_{i = 2}^p$. 
The threshold function can be expressed by
\[
\tau(\rho) = \sqrt{ 4 \left( \rho^2 \sum_{i = j+1}^{p} \lambda_i^{-2} + j\right) \log \frac{p-1}{\alpha} } + 2 \log\left(\frac{p-1}{\alpha}\right).
\]
Define the following quantities,
\[
\begin{aligned}
&A = 4 \log((p-1)/\alpha) \sum_{i = j+1}^{p} \lambda_i^{-2}, \quad B = 4j \log((p-1)/\alpha)  \\
&D = 2 \log((p-1)/\alpha), \quad E = \sum_{i = j+1}^p \frac{(\ub_i^\top \yb)^2 - 1}{\lambda_i}, \\
&F = \sum_{i=2}^j ((\ub_i^\top \yb)^2 - 1).
\end{aligned}
\]
Then we reject iff 
\[
\tau(\rho) = \sqrt{\rho^2 A + B} + D < \rho E + F = \hat t(\rho).
\]
Notice that $A,B,D > 0$, so $\tau(\rho)$ has strictly positive curvature and is convex.
Thus, $\tau(\rho) - \hat t(\rho)$ is convex within $\lambda_j \le \rho \le \lambda_{j+1}$ and has a unique minimum.
We can minimize the unrestricted function,
\[
\rho^* = \arg \min_\rho \sqrt{\rho^2 A + B} + D - \rho E - F
\]
and we find that this is attained at
\[
\rho^* = \left\{
\begin{array}{ll}
0, &E^2 \ge A\\
\sqrt{\frac{E^2 B}{A^2 - E^2 A}}, &\textrm{otherwise}.
\end{array}
 \right.
\]
We know by convexity that if $\rho^* < \lambda_j$ then the constrained maximum is attained at $\lambda_j$, and if $\rho^* > \lambda_{j+1}$ then it is attained at $\lambda_{j+1}$. 
For each $j$, we can construct $A,B,D,E,F$ and define
\[
\rho_j = \left\{
\begin{array}{ll}
\lambda_j, &E^2 \ge A \textrm{ or } \sqrt{\frac{E^2 B}{A^2 - E^2 A}} \le \lambda_j\\
\lambda_{j+1}, &\sqrt{\frac{E^2 B}{A^2 - E^2 A}} \ge \lambda_{j+1}\\
\sqrt{\frac{E^2 B}{A^2 - E^2 A}}, &\textrm{otherwise}.
\end{array}
 \right.
\]
Then the following proposition holds,
\begin{proposition}
\label{prop:discrete}
The adaptive GFSS test rejects $H_0$ if and only if
\[
\exists j \in \{2,\ldots,p\}, \quad \tau(\rho_j) < \hat t(\rho_j).
\]
\end{proposition}

This proposition has theoretical implications as well as practical.
It shows us that we only need to provide a theoretical control of $p$ separate GFSS values.
We see that Proposition \ref{prop:discrete} was foreshadowed by the specific form of $\tau(\rho)$ in \eqref{eq:tau_function}.
The clever choice of threshold function $\tau(\rho)$ naturally gives us a control on the false alarm (type 1 error).

\begin{theorem}
\label{thm:adaptive}
The probability of false rejection (type 1 error) is bounded by
\[
\sup_{\PP_0 \in H_0} \PP_0 \{ \exists \rho , \hat t (\rho) > \tau (\rho) \} \le \alpha.
\]
Consider models from the alternative hypotheses, $H_1^{PC}, H_1^{S}$ as functions of $\rho$.
Let $\rho^*$ be the smallest such $\rho^*$ such that $\xb + \epsilonb$ is contained in the alternative hypotheses.
Then the probability of type 2 error is bounded by $\gamma > 0$ if
\[
\begin{aligned}
\tau(\rho^*) < &\frac{\mu^2}{2 \sigma^2} - 2 \frac\mu\sigma \sqrt{2\log(2 / \gamma)}\\
& - 2 \sqrt{\sum_{i=2}^p \min \big\{1 , \frac{{\rho^*}^2}{\lambda_i^2} \big\} \log(2/\gamma)}.
\end{aligned}
\]
\end{theorem}

The interpretation is that by providing the thresholding function $\tau(\rho)$ we are in effect thresholding at $p$ distinct points which can be controlled theoretically by union bounding techniques.
The following corollary describes the SNR rates necessary for risk consistency.

\begin{corollary}
The adaptive GFSS asymptotically distinguishes $H_0$ from $H_1^{PC},H_1^S$ if
\[
\frac{\mu}{\sigma} = \omega \left( \sqrt{\sum_{i=2}^p \min \big\{1 , \frac{{\rho^*}^2}{\lambda_i^2} \big\} \log p} + \log p \right)^{1/2}.
\]
\end{corollary}

So we are able to make all the same theoretical guarantees with the adaptive GFSS as the GFSS with an additional multiplicative term $(\log p)^{1/4}$ and an additive term of $(\log p)^{1/2}$.
We will now show how this theory is applicable by developing corollaries for different specific graph topologies.

\section{Specific Graph Models and Experiments}

In this section, we demonstrate the power and flexibility of Theorem \ref{thm:main} by analyzing in detail the performance of the GFSS over a simulated As detection example and three important graph topologies: balanced binary trees, the torus graph and Kronecker graphs (see \cite{leskovec2007scalable,leskovec2010kronecker}).
The explicit goals of this section are as follows: 
\begin{enumerate}
\item Demonstrate the effectiveness of the GFSS on partially simulated dataset from the Arsenic detection graph.
\item Determine the implications of Theorem \ref{thm:main} in these specific graph examples
for some example signal classes;
\item Demonstrate the competitiveness of the GFSS and the adaptive GFSS against the aggregate and max statistics;
\item Provide an example of the general graph structure;
\end{enumerate}

\subsection{Arsenic Detection Simulation}

In order to compare the GFSS, adaptive GFSS, and the naive estimators, we construct realistic signals over the Arsenic graph (so that we have a ground truth) and generate Gaussian noise over these signals.
This will also provide us with an opportunity to make some practical recommendations on how to use the GFSS.
In order to construct realistic signals, we will use the locations of the principle aquifers in Idaho \cite{USGS2003principal}.
We will associate an As test well with its closest aquifer and select randomly a small number of aquifers that we will consider to be contaminated.
The signal, $x_i$, that we will construct is zero over all of the wells, $i$, not belonging to a contaminated aquifer and elevated over those wells that are.
Specifically, we set the level of elevation in each simulation such that $\| \xb \|^2 = 5$ and we generate additive Gaussian noise with $\sigma = 1$.

In the first experiment (Figure \ref{fig.5} left), we chose $3$ aquifers at random which resulted in $112$ contaminated wells.
By our choice of $\| \xb \|^2 = 5$, the signal size at each contaminated well was $0.47$ which is substantially less than the noise level $\sigma$.
In the second experiment (Figure \ref{fig.5} middle), we chose $1$ of the larger aquifers (with greater than $100$ wells) at random which resulted in $109$ contaminated wells.
In the third experiment (Figure \ref{fig.5} right), we selected $1$ of the somewhat smaller aquifer (with greater than $50$ wells) at random which resulted in $62$ contaminated wells.
We simulate the probability of correct detection (rejecting $H_0$ when the truth is $H_1$) versus the probability of false alarm (falsely rejecting $H_0$) by making $1000$ draws from the noise distribution (with and without the signal $\xb$ for $H_1$ and $H_0$ respectively).

\begin{figure*}[!htbp]
\centering
\mbox{
\subfigure{\includegraphics[width=2.1in]{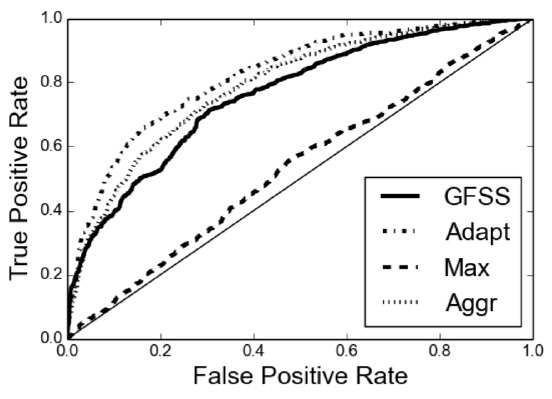}}
\subfigure{\includegraphics[width=2.1in]{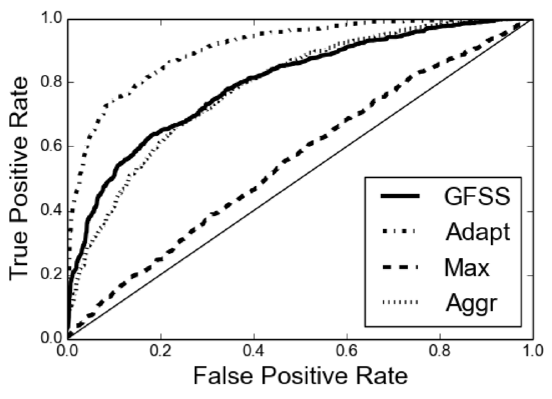}}
\subfigure{\includegraphics[width=2.1in]{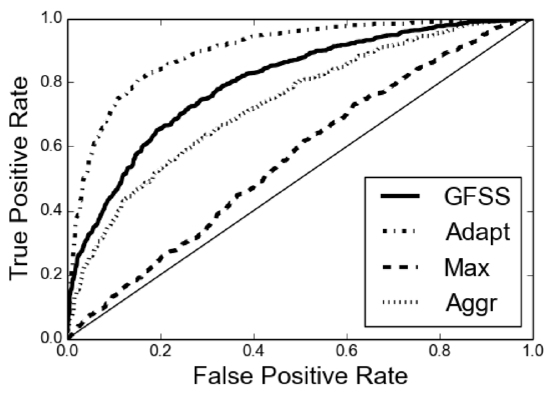}}
}
\caption{\small{\textbf{(Arsenic Contamination Simulations)} Simulations of the size (false positive rate) and the power under $H_1^{PC}$ for the As simulations of the GFSS, adaptive GFSS (Adapt), Max statistic (Max), and Aggregate statistic (Aggr).  The figures are for 3 contaminated aquifers (left), 1 large contaminated aquifer (middle), and 1 smaller contaminated aquifer (right).}}
\label{fig.5}
\end{figure*}

As can be seen the adaptive GFSS test strictly outperforms all of the test statistics, which demonstrates the importance of adapting to the $\rho$ parameter.  
Moreover, the GFSS with the somewhat arbitrary choice of $\rho = \lambda_{109}$ begins to outperform the Aggregate statistic for the smaller contamination as our theory predicts.
The adaptive GFSS is a substantially better alternative to an arbitrary choice of $\rho$ and statistics that do not take the kNN graph structure into account.

While Theorem \ref{thm:main} can be inverted to obtain a P-value that is valid for finite $p$, this may be too conservative for practical purposes.
We recommend one of two approaches: forming a Z-score that is asymptotically normal under $H_0$, and using a permutation test.
Under $H_0$, the GFSS, $\hat t$, has zero mean and because it is the sum of weighted $\chi^2_1$ random variables it has a variance of $2(\sum_{i=2}^p h(\lambda_i)^4)$.
Thus, a Z-score can be calculated by $\hat Z = \hat t / \sqrt{2 \sum_{i=2}^p h(\lambda_i)^4}$,
which can be shown to have an asymptotic standard normal distribution under some regularity conditions.
Thus, we can form an asymptotically valid P-value by applying the standard normal inverse CDF to $\hat Z$.
While this is valid when the noise is Gaussian, in many instances the measurements are not Gaussian and we interpret $H_0$ to mean that $x_i = \EE y_i$ is constant over the graph which is a weaker assumption (recall we had binary observations in section I.A, but we used the GFSS none-the-less).
In this case, we can apply a permutation test, by which we randomly permute the coordinates of $\yb$ and maintain the graph structure.
We interpret the resulting statistic $\hat t$ as a simulation of the GFSS under $H_0$.
Then an estimated P-value would be the fraction of permutations that have a larger $\hat t$ then the actual GFSS.
This was used to construct the reported P-values in Section I.B.

\subsection{Balanced Binary Trees}

Balanced trees are graph structures of particular interest because they provide a simple hierarchical structure.
Furthermore, the behavior of the graph spectra for the balanced binary tree provides a natural multiscale basis \cite{singh2010detecting,sharpnack2010identifying}.
We begin this analysis of the GFSS by applying it to the balanced binary tree (BBT) of depth $\ell$.
We consider the class of signals defined by $\rho = [cp^\alpha(1 - cp^{\alpha-1})]^{-1}$
where $0<c\le 1/2, 0<\alpha\le 1$. 
This class is interesting as it includes, among others, clusters of constant signal which are subtrees of size at least $c p^\alpha$ (subtrees can be 
isolated from a tree by cutting a single edge and hence have cut size $1$). 

\begin{corollary}
\label{cor:BBT}
Let $G$ be a balanced binary tree with $p$ vertices, and let
$\rho = p[cp^\alpha(p - cp^\alpha)]^{-1}$.\newline
(a) The GFSS can asymptotically distinguish $H_0$ from signals 
within $H_1^{PC}, H_1^S$ if the SNR is stronger than
\[
\frac{\mu}{\sigma} = \omega ( p^\frac{1-\alpha}{4} (\log p)^{1/4} ).
\]
(b) The adaptive GFSS distinguishes the hypotheses of (a) if 
\[
\frac{\mu}{\sigma} = \omega ( p^\frac{1-\alpha}{4} (\log p)^{1/2}).
\]
(c) $H_0$ and $H_1^{PC}$ are asymptotically indistinguishable if 
\[
\frac{\mu}{\sigma} = o ( p^\frac{1-\alpha}{4} ).
\]
\end{corollary}

The conclusion is that for the BBT the GFSS and the adaptive GFSS is near optimal with respect to critical SNR.
The proof (Appendix A) is based on the special form of the spectrum of the BBT.
So in this case, the GFSS consistently dominates the naive statistics and the theoretical results are very close to the lower bounds for any $\alpha$.

We simulate the probability of correct detection versus the probability of false alarm.
These are given for the four statistics in Figure \ref{fig1} 
as the test threshold, and hence the probability of false alarm, is varied. 
The GFSS is computed with the correct $\rho$, which is in general unknown.
Different statistics dominate under different choices of cluster size parameter, $\alpha$.
When $\alpha = 1$, corresponding to large clusters, where the size is on the same order as $p$, the aggregate statistic is competitive with the adaptive statistic.
When $\alpha = 0.5$, corresponding to clusters of size $\asymp p^{1/2}$, the aggregate becomes less competitive and the max more competitive than the $\alpha = 1$ case, and the GFSS remains the dominating test.
In each case, we set $c = 1/2$, which ensures that the $\alpha = 1$ case does not select the entire tree.


\begin{figure*}[!htbp]
\centering
\mbox{
\subfigure{\includegraphics[width=2.1in]{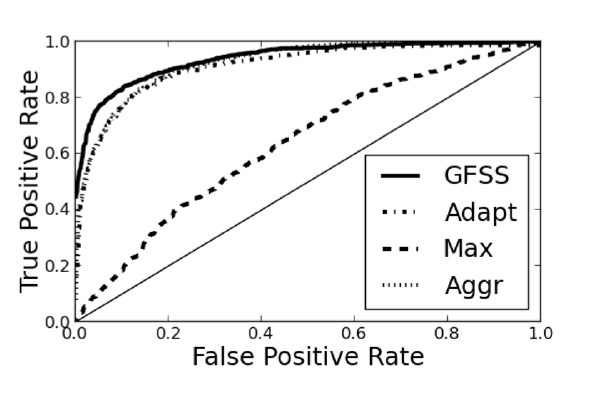}}
\subfigure{\includegraphics[width=2.1in]{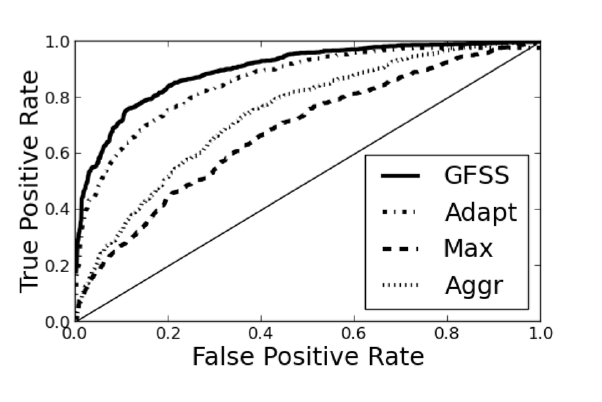}}
}
\caption{\small{\textbf{(BBT Comparisons)} Simulations of the size (false positive rate) and the power under $H_1^{PC}$ for the balanced binary tree of the GFSS, adaptive GFSS (Adapt), Max statistic (Max), and Aggregate statistic (Aggr).  The figures are for the tree of depth $\ell = 6$, $p = 2^{\ell + 1} - 1 = 127$, with choice of $\alpha = 1$ (left) and $\alpha = 0.5$ (right).}}
\label{fig1}
\end{figure*}

\subsection{Torus Graph}

The torus has been a pedagogical example, but it is also an important example as it models a mesh of sensors in two dimensions.
We will analyze the performance guarantees of the GFSS over our running example, the 2-dimensional torus graph with $\ell$ vertices along each dimension ($p = \ell^2$).
To include squares of size $p^{1 - \beta}$, as in the examples, then we would obtain $\rho \asymp p^{-(1 - \beta)/2}$.
The following result is due to a detailed analysis of the spectrum of the torus.

\begin{corollary}
\label{cor:torus}
Let $G$ be the $\ell \times \ell$ square torus ($p = \ell^2$), and let $\rho = c p^{-(1 - \beta)/2}$ for $\beta \in [0,1)$. \newline
(a) The GFSS can asymptotically distinguish $H_0$ from $H_1^{PC},H_1^S$ if the SNR satisfies
\[
\frac{\mu}{\sigma} = \omega ( p^{\frac{3}{20} + \frac{1}{10} \beta} ).
\]
(b) The adaptive GFSS can asymptotically distinguish the hypotheses of (a) if
\[
\frac{\mu}{\sigma} = \omega ( p^{\frac{3}{20} + \frac{1}{10} \beta} (\log p)^{1/4} ).
\]
(c) $H_0$ and $H_1^{PC}$ are asymptotically indistinguishable if the SNR is weaker than
\[
\frac{\mu}{\sigma} = o ( p^{\frac{\beta}{4}} ).
\]
\end{corollary}

The implication of Cor.~\ref{cor:torus} is that when $\beta > 0$ (the clusters are not too large), the GFSS is consistent under an SNR lower than $p^{1/4}$.
Regardless of the $\beta$ parameter the GFSS never achieves the lower bound for the torus graph, which suggests an approach that exploits the specific structure of the torus may yet outperform the GFSS.
We simulate the performance of the test statistics over a $30 \times 30$ torus, with $\beta = 0,.5,.75$ with $c = 1/2$
When $\beta$ is small (large clusters), we suffer an additional factor of $p^{3(1-\beta)/20}$ in the upper bound.
Despite the theoretical shortcomings of in this case, the simulations (Figure \ref{fig3}) suggest that the GFSS is significantly superior to the naive tests for medium sized clusters.

\begin{figure*}[!htbp]
\centering
\mbox{
\subfigure{\includegraphics[width=2.1in]{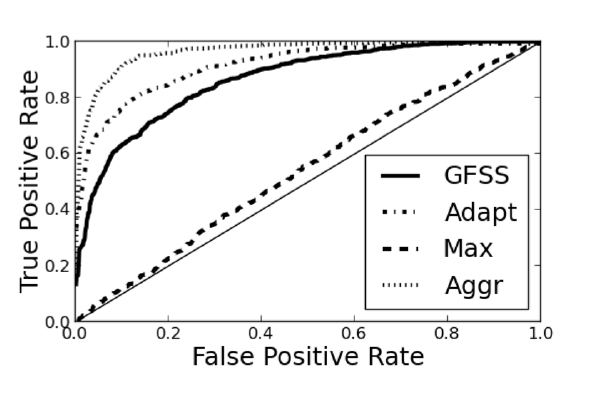}}
\subfigure{\includegraphics[width=2.1in]{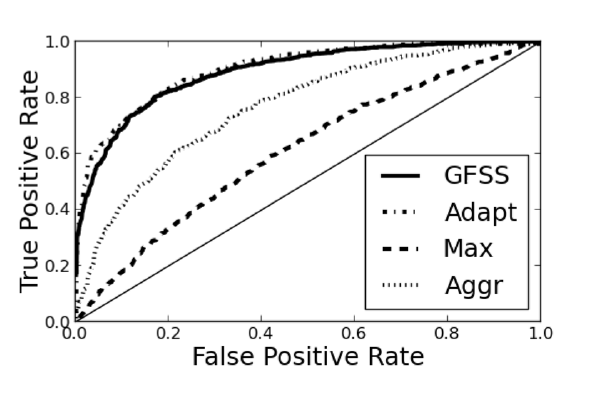}}
\subfigure{\includegraphics[width=2.1in]{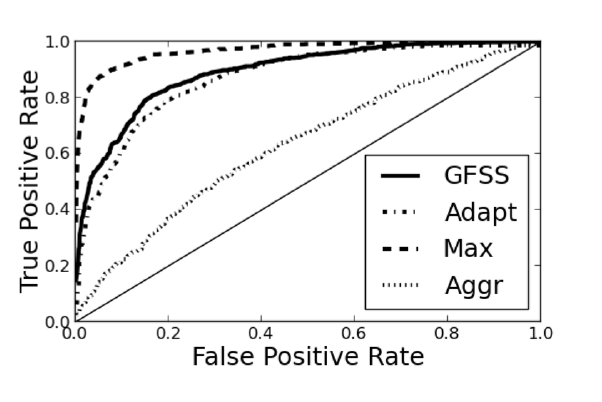}}
}
\caption{\small{\textbf{(Torus Comparisons)} Simulations of the size (false positive rate) and the power under $H_1^{PC}$ for the Torus of the GFSS, adaptive GFSS (Adapt), Max statistic (Max), and Aggregate statistic (Aggr).  The figures are for side length of $\ell = 30$, $p = \ell^2 = 900$, with choice of $\beta = 0$ (top left), $\beta = .5$ (top right) and $\beta = .75$ (bottom).}}
\label{fig3}
\end{figure*}

\subsection{Kronecker Graphs}

Much of the research in complex networks has focused on observing statistical phenomena that is common across many data sources.
The most notable of these are that the degree distribution obeys a power law (\cite{faloutsos1999power}) and networks are often found to have small diameter (\cite{milgram1967small}).
A class of graphs that satisfy these, while providing a simple modelling platform are the Kronecker graphs (see \cite{leskovec2007scalable,leskovec2010kronecker}).
Let $H_1$ and $H_2$ be graphs on $p_0$ vertices with Laplacians $\Delta_1, \Delta_2$ and edge sets $E_1, E_2$ respectively.
The Kronecker product, $H_1 \otimes H_2$, is the graph over vertices $[p_0] \times [p_0]$ such that there is an edge $((i_1,i_2),(j_1,j_2))$ if $i_1 = j_1$ and $(i_2,j_2) \in E_2$ or $i_2 = j_2$ and $(i_1,j_1) \in E_1$.
We will construct graphs that have a multi-scale topology using the Kronecker product.
Let the multiplication of a graph by a scalar indicate that we multiply each edge weight by that scalar.
First let $H$ be a connected graph with $p_0$ vertices.
Then the graph $G$ for $\ell > 0$ levels is defined as  
\[
\frac{1}{p_0^{\ell-1}} H \otimes \frac{1}{p_0^{\ell-2}} H \otimes ... \otimes \frac{1}{p_0} H \otimes H.
\]
The choice of multipliers ensures that it is easier to make cuts at the more coarse scale.
Notice that all of the previous results have held for weighted graphs.

\begin{corollary}
\label{cor:kron}
Let $G$ be the Kronecker product of the base graph $H$ described above with $p = p_0^\ell$ vertices, and let $\rho \asymp p_0^{2k - \ell - 1}$ (which includes cuts within the $k$ coarsest scale).\newline
(a) The GFSS can asymptotically distinguish $H_0$ from signals from $H_1^{PC},H_1^S$ if the SNR is stronger than
\[
\frac{\mu}{\sigma} = \omega ( p^{k/2\ell} (\textrm{diam}(H))^{1/4}),
\]
where $\textrm{diam}(H)$ is the diameter of the base graph $H$.\newline
(b) The adaptive GFSS can distinguish the hypotheses of (a) if
\[
\frac{\mu}{\sigma} = \omega ( p^{k/2\ell} (\textrm{diam}(H) \log p)^{1/4}).
\]
(c) $H_0$ and $H_1^{PC}$ are asymptotically indistinguishable if 
\[
\frac{\mu}{\sigma} = o( p^{k/4 \ell} ).
\]
\end{corollary}

The proof and an explanation of $\rho$ is in the appendix.
The implication of Cor.~\ref{cor:kron} is that only for $k$ small is the GFSS nearly optimal.
Generally, one will suffer a multiplicative term of $p^{k/4\ell}$.
As we can see from the simulations the $k = 1$ case is exactly when the aggregate statistic dominates (see Figure \ref{fig5}).
When $1 < k < \ell$, the GFSS improves on the aggregate and the max statistics.
Throughout these simulations we set $\rho = p_0^{2k-\ell-1}$.

\begin{figure*}[!htbp]
\centering
\mbox{
\subfigure{\includegraphics[width=2.1in]{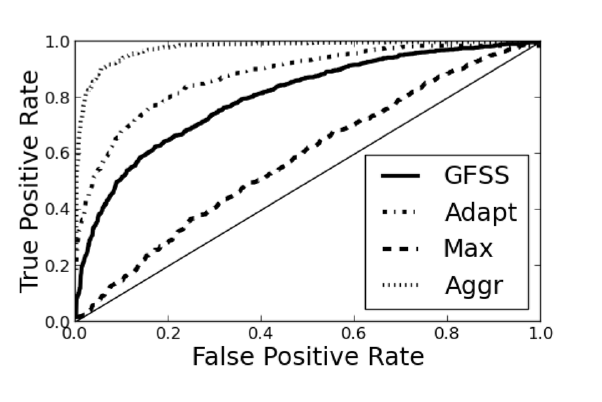}}
\subfigure{\includegraphics[width=2.1in]{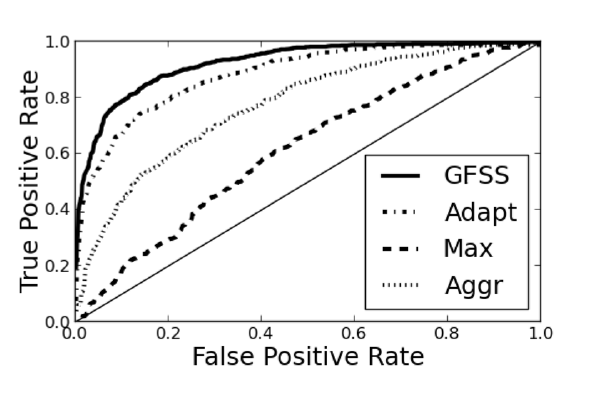}}
}
\caption{\small{\textbf{(Kronecker Comparison)} Simulations of the size (false positive rate) and the power under $H_1^{PC}$ for the Kronecker graph of the GFSS, adaptive GFSS (Adapt), Max statistic (Max), and Aggregate statistic (Aggr).  The figures are for a base graph of size $p_0 = 6$ and Kronecker power of $\ell = 3$, so $p = p_0^\ell = 216$.  The cuts were chosen at the coarsest scale, $k = 1$, (left) and at the second coarsest, $k = 2$ (right).}}
\label{fig5}
\end{figure*}


One may rightly ask if the gap between the upper bounds (Corollaries \ref{cor:BBT} (b), \ref{cor:torus} (b), \ref{cor:kron} (b)) and the lower bounds (Corollaries \ref{cor:BBT} (c), \ref{cor:torus} (c), \ref{cor:kron} (c)) is just due to a lack of theoretical know-how and the test is actually optimal.
We attempt to assess this concern by plotting the performance of the GFSS with the SNR increasing according to the scaling dictated by the upper bounds (Figure \ref{fig4}).
For the BBT because the curve does not change significantly with $p$ (as the tree depth $l$ increases), the upper bound is supposed to be tight.
In the torus graph, for large rectangles ($\beta = 0$) the upper bound appears to be correct, while for moderately sized rectangles ($\beta = .5$) there may be a gap between our theoretical bound, \ref{cor:torus} (b), and the actual performance of GFSS.
For the Kronecker graph there appears to be a gap for both scalings ($k=1$ and $k=2$) of cluster size, indicating that the performance of the GFSS may be better than predicted by our theory.

\begin{figure*}[!htbp]
\centering
\mbox{
\subfigure{\includegraphics[width=2.1in]{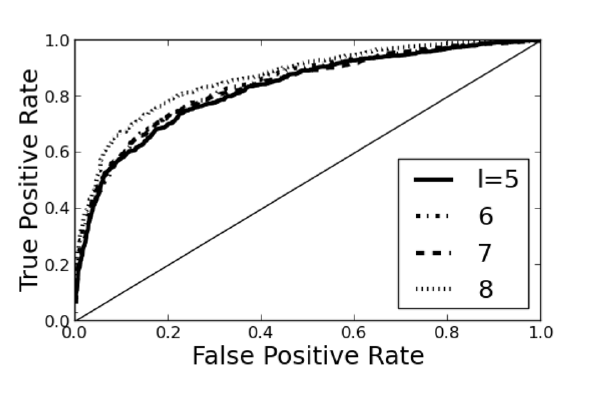}}
\subfigure{\includegraphics[width=2.1in]{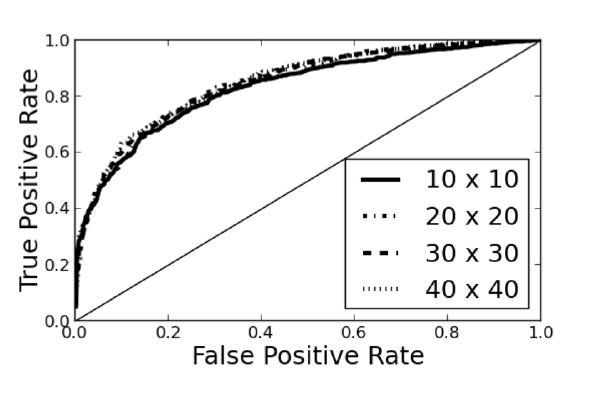}}
\subfigure{\includegraphics[width=2.1in]{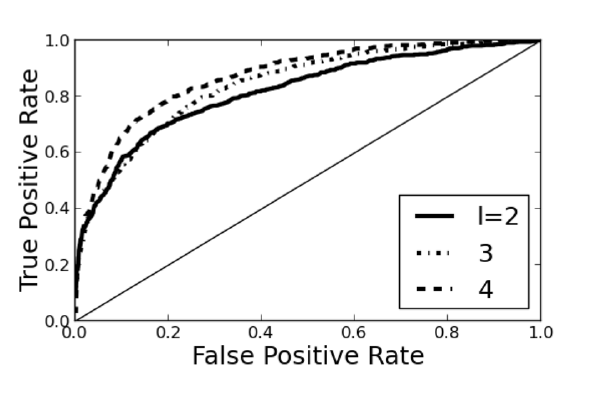}}
}
\mbox{
\subfigure{\includegraphics[width=2.1in]{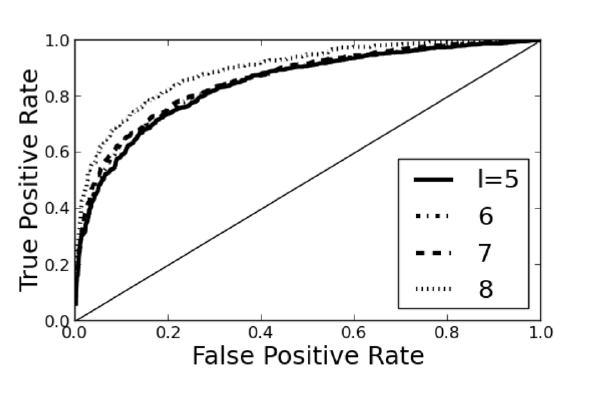}}
\subfigure{\includegraphics[width=2.1in]{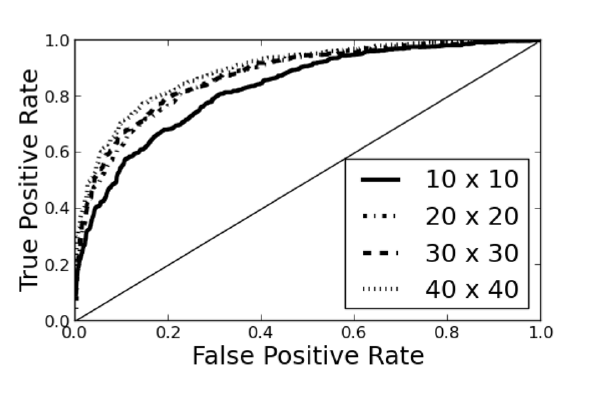}}
\subfigure{\includegraphics[width=2.1in]{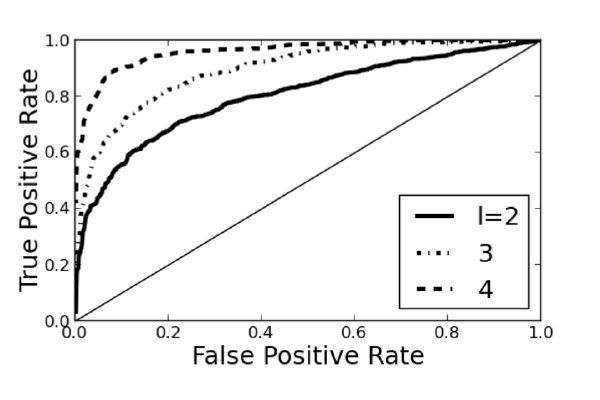}}
}
\caption{\small{\textbf{(Rescaling by Theoretical Bounds)} The size (false positive rate) and power (true positive rate) of the GFSS as $p$ increases for the following graph and signal models: BBT with $\alpha = 1$ (top left) and $\alpha = .5$ (bottom left); Torus with $\beta = 0$ (top middle) and $\beta=.5$ (bottom middle); Kronecker graph with base graph size $p_0 = 6$ and $k = 1$ (top right) and $k = 2$ (bottom right).  The SNR was allowed to scale according to Cor.~\ref{cor:BBT} (b) (left), Cor.~\ref{cor:torus} (b) (middle), Cor.~\ref{cor:kron} (b) (right).}}
\label{fig4}
\end{figure*}

\subsection{General Graph Structure, $H_1^S$}

The piecewise constant alternative hypothesis $H_1^{PC}$ is amenable to a sophisticated theoretical analysis and it motivates the GFSS.
Unfortunately, it is very easy to modify signals in $\Xcal_{PC}(\mu,\rho)$ by slight perturbations and find a signal that is outside our supposed class.
This lack of robustness is rightly alarming, and it is through the general graph structured class, $\Xcal_S(\mu,\rho)$, that we intended to include these perturbations.
We now provide a signal subsampling scheme that will demonstrate the performance of the GFSS under signal perturbations.

Suppose that we begin with a signal $\xb \in \Xcal_{PC}(\mu,\rho)$ such that $\xb = \delta \one_C$ and modify it in the following way: let $C' \subset C$ and make $\xb' \propto \one_C'$ such that $\xb' \in \Xcal_S(\mu,\rho)$.
We now determine the normalization that would make this so.
Notice that
\[
\left| \frac{\xb'^\top \one_C}{|C|} - \frac{\xb'^\top \one_{\bar C}}{|\bar C|} \right| \sqrt{\frac{|C||\bar C|}p} \ge \mu.
\]
Hence, $\xb' = \delta' \one_{C'}$ implies that $\delta' = \delta |C| / |C'|$ is sufficient.
So for the subsampled signal $\xb'$ to remain in $\Xcal_S(\mu,\rho)$ we will need to boost the signal by a factor of $|C|/|C'|$.
Figure \ref{fig7} shows the performance curves for the GFSS when the signal cluster $C'$ is formed by including each vertex in $C$ according to independent Bernoulli($q$) random variables for the BBT.
To make the comparison fair we boost the signal according to the above formulation.
As one can see the subsampling does not make the performance worse.

\begin{figure*}[!htbp]
\centering
\mbox{
\subfigure{\includegraphics[width=2.1in]{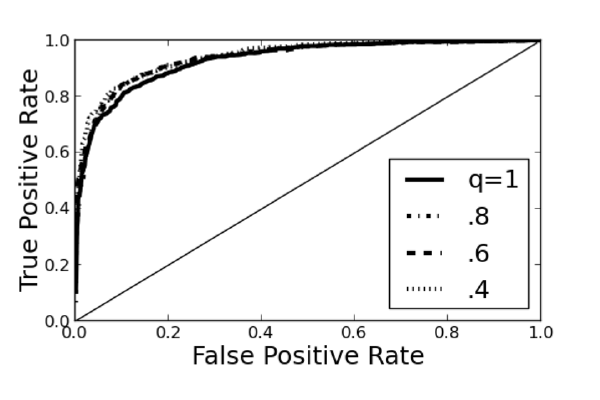}}
\subfigure{\includegraphics[width=2.1in]{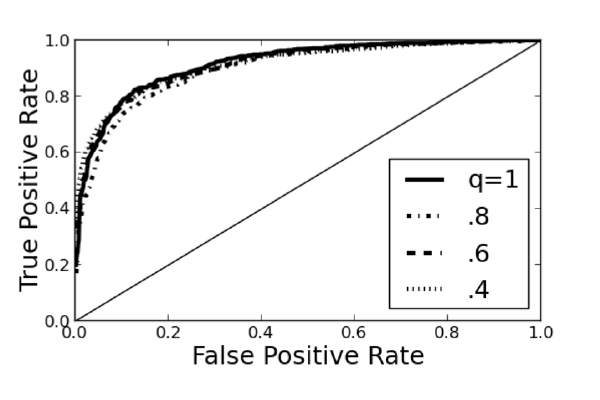}}
}
\caption{\small{\textbf{(BBT Perturbations)} Simulations of the size (false positive rate) and the power under $H_1^{S}$ for the balanced binary tree of the GFSS with changing cluster sampling probability, $q$.  The figures are for the tree of depth $\ell = 6$, $p = 2^{\ell + 1} - 1 = 127$, with choice of $\alpha = 1$ (left) and $\alpha = 0.5$ (right).}}
\label{fig7}
\end{figure*}

\section{Conclusion}

We studied the problem of how to tractably detect anomalous activities in networks under Gaussian noise.
We outlined what is known regarding the performance of the agglomerative and maximum statistics.
These statistics do not take into account the graph structure and we wish instead to exploit the structure of the graph to obtain a superior test statistic.
To this end we developed the graph Fourier scan statistic, suggesting it as a computationally feasible alternative to the GLRT.
We completely characterized the performance of the GFSS for any graph in terms of the spectrum of the combinatorial Laplacian.
The theoretical analysis of the GFSS naturally led to the development of the adaptive GFSS.

We applied the main result to three graph models: balanced binary trees, the lattice and Kronecker graph.
We demonstrated that the performance is not specific to piecewise constant signals, and we are able to extend our results to a more general alternative class, $H_1^S$.
We demonstrated that though the theoretical performance of the GFSS for the Torus graph and Kronecker graph may be sub-optimal, there is experimental evidence to indicate that this is partly an artifact of the theoretical analysis technique.
We see that not only is it statistically sub-optimal to ignore graph structure, but in many of these cases the GFSS gives a near optimal performance.


\bibliographystyle{alpha}
\bibliography{biblio}

\newpage
\appendix
\appendix
\section{Proofs}

The result now follows by considering all the indicator functions corresponding to the sets in $\mathcal{C}$. 

\begin{proof}[Proof of Proposition \ref{prop:GFSS_def}]
    To prove the claim we will first rewrite the SSS in an equivalent but more
    convenient form which we will then bound from above and below using the GFSS. 
 To this end we recall the arguments from Lemma 7 of \cite{sharpnack2012changepoint}.
 Since $G$ is connected, the combinatorial Laplacian $\Delta$ is symmetric, its
 smallest eigenvalue is zero and the remaining eigenvalues are positive. By the
 spectral theorem, we can write $\Delta = \Ub \Lambda \Ub^\top$, where $\Lambda$ is
 a $(p-1) \times (p-1)$ diagonal matrix containing the positive eigenvalues of
 $\Delta$, $\lambda_2,\ldots, \lambda_p$, in increasing order. The columns of the
 $p \times (p-1)$ matrix $\Ub$ are the associated eigenvectors.
Then, since each vector $\xb \in \mathbb{R}^p$ with ${\bf1}^\top \xb = 0$ can be written as $\Ub \zb$ for a unique vector $\zb \in \mathbb{R}^{p-1}$, we have
\[
\begin{array}{rcl}
\mathcal{X} & = & \{ \xb \in \mathbb{R}^p \colon \xb^\top \Delta \xb \leq \rho, \xb^\top \xb =1, {\bf1}^\top \xb \leq 0\}\\
& = & \{ \Ub \zb \colon \zb \in \mathbb{R}^{p-1}, \\
&&\quad \zb^\top \Ub^\top \Delta \Ub \zb \leq \rho, \zb^\top \Ub^\top \Ub\zb \leq1\} \\
& = & \{ \Ub \zb \colon \zb \in \mathbb{R}^{p-1}, \frac{1}{\rho} \zb^\top \Lambda \zb \leq 1, \zb^\top \zb \leq1\}, \\
\end{array}
\]
where in the third identity we have used the fact that $\Ub^\top \Ub =
\Ib_{p-1}$. Letting $\mathcal{Z} = \{ \zb \in \mathbb{R}^{p-1} \colon
\frac{1}{\rho} \zb^\top \Lambda \zb \leq 1, \zb^\top \zb \leq 1 \}$, we see that
the SSS can be equivalently expressed as
\begin{equation}\label{eq:shat.simplified}
    \sqrt{\widehat{s}} =  \sup_{\xb \in \mathcal{X}} \xb^\top \yb = \sup_{\zb \in
    	\mathcal{Z}} \zb^\top \Ub^\top \yb. 
    \end{equation}

Next, let  $\Ab = \frac{1}{\rho} \Lambda = \textrm{diag}\{ a_i\}_{i=1}^{p-1}$,
where $a_i = \lambda_{i+1}/\rho$, for $i=1,\ldots,p-1$.
If $\zb \in \mathbb{R}^{p-1}$ satisfies  $\| \zb \| \le 1$ and $\zb^\top \Ab \zb
\le 1$, then
\[
    \sum_{i =1}^p \max\{1, a_i\} z_i^2 \le \| \zb \|^2 + \zb^\top \Ab \zb \le 2.
    \]
Similarly, if $\sum_{i =1}^p \max\{1, a_i\} z_i^2  \le 1$, then we must have $
\max \left\{ \|
\zb \|, \zb^\top \Ab \zb \right\} \le 1$ as well. 
Now let $\Ab'$ be the $(p-1)$-dimensional diagonal  matrix with entries $\max \{1,
a_i \}$, $i=1,\ldots,p-1$ and set $\mathcal{Z}_1 = \{ \zb \in \mathbb{R}^{p-1}
\colon \zb^\top \Ab' \zb \leq 1 \}$ and  $\mathcal{Z}_2 = \{ \zb \in \mathbb{R}^{p-1}
\colon \zb^\top \Ab' \zb \leq 2 \}$.  Thus we have shown  that 
\[
\mathcal{Z}_1 \subset \mathcal{Z} \subset \mathcal{Z}_2. 
    \]
Using \eqref{eq:shat.simplified}, the previous inclusions imply the following
bounds on the square root of the SSS:
\[
    \sup_{\zb \in \mathcal{Z}_1} \zb^\top \Ub^\top \yb \leq
    \sqrt{\widehat{s}} \leq \sup_{\zb \in \mathcal{Z}_2} \zb^\top \Ub^\top 
    \yb 
    \]
which in turn are equivalent to the bounds
\[
    \sup_{\{ \zb \in \mathbb{R}^p \colon \zb^\top \Ub \Ab' \Ub^\top \zb \le 1 \}}  \yb^\top \zb \le \sqrt{\hat s}
    \le \sup_{\{ \zb \in \mathbb{R}^p \colon \zb^\top \Ub \Ab' \Ub^\top \zb \le 2\}}  \yb^\top \zb,
\]
since every $\zb \in \mathbb{R}^{p-1}$ can be written as $\Ub^\top \zb$ for some
$\zb \in \mathbb{R}^p.$\footnote{In fact, $\zb = \Ub^\top \zb_1 =  \Ub^\top \zb_2
$ if and only if the difference $\zb_1 - \zb_2$ belongs to the linear  subspace of
$\mathbb{R}^p$ spanned by the constant vectors.}
 
All that remains is to show that 
\[
\hat t = \sup_{\{ \zb \in \mathbb{R}^p \colon \zb^\top \Ub \Ab' \Ub^\top \zb \le 1 \}} \yb^\top \zb.
\]
This can be seen by strong duality for convex programs,
\[\begin{aligned}
\sup_{\{ \zb \in \mathbb{R}^p \colon \zb^\top \Ub \Ab' \Ub^\top \zb \le 1 \}} \yb^\top \zb = \sup_{\{ \zb \in \mathbb{R}^p \colon \zb^\top \Ab' \zb \le 1 \}} (\Ub^\top \yb)^\top \zb \\
=  \sup_{\{ \zb \in \mathbb{R}^{p-1} \}} \inf_{\eta \ge 0} (\Ub^\top \yb)^\top \zb - \eta (\zb^\top \Ab' \zb - 1)\\
= \inf_{\eta \ge 0} \sup_{\{ \zb \in \mathbb{R}^{p-1} \}} (\Ub^\top \yb)^\top \zb - \eta (\zb^\top \Ab' \zb - 1).
\end{aligned}
\]
The solution to the maximization problem is $\zb = (2 \eta \Ab')^{-1} (\Ub^\top \yb)$, and plugging this in it becomes
\[\inf_{\eta \ge 0} (\Ub \yb)^\top (4 \eta \Ab')^{-1} (\Ub^\top \yb) + \eta
\]
which is minimized at 
\[
\eta = \sqrt{ (\Ub \yb)^\top (4 \Ab')^{-1} (\Ub \yb) }.
\]
Plugging this in completes our proof.
\vspace{-.25cm}
\end{proof}
\vspace{-.1cm}

\begin{proof}[Proof of Theorem \ref{thm:main}]
We will use the following lemma regarding the concentration of $\chi^2$ random variables.
\begin{lemma}[\cite{laurent2000adaptive}]
\label{lem:chi_squared}
Let for $i \in \{2,\ldots,p \}$, $a_i \ge 0$ and $\{X_i \}_{i = 1}^p$ be independent $\chi^2_1$ random variables. 
Define $Z = \sum_{i = 1}^p a_i (X_i - 1)$
\[
\begin{aligned}
\PP \{ Z \ge 2 \| \ab \|_2 \sqrt{x} + 2 \| \ab \|_\infty x \} \le e^{-x} \\
\PP \{ Z \le - 2 \| \ab \|_2 \sqrt{x} \} \le e^{-x} 
\end{aligned}
\]
\end{lemma}

Recall the notation of the proof of Prop.~\ref{prop:GFSS_def}.
The probability of error under the null, \eqref{eq:null_control}, follows from Lemma \ref{lem:chi_squared}.
Consider any of the alternatives, then $\hat t$ can be written,
\[
\begin{aligned}
&\hat t = \yb^\top \Ub (\Ab')^{-1} \Ub^\top \yb - \tr (\Ab')^{-1} \\
&= \xb^\top \Ub (\Ab')^{-1} \Ub^\top \xb + 2 \xb^\top \Ub (\Ab')^{-1} \Ub^\top \epsilonb \\
&\quad + \epsilonb^\top \Ub (\Ab')^{-1} \Ub^\top \epsilonb- \tr (\Ab')^{-1}\\
&\overset{d}{=} \xb^\top \Ub (\Ab')^{-1} \Ub^\top \xb + 2 \xb^\top \Ub (\Ab')^{-1} \epsilonb + \epsilonb^\top (\Ab')^{-1} \epsilonb- \tr (\Ab')^{-1}\\
\end{aligned}
\]
where $\overset{d}{=}$ denotes equality in distribution (which follows from rotational invariance of the isonormal Gaussian).
By Gaussian concentration, with probability at least $1 - \alpha$,
\[
\xb^\top \Ub (\Ab')^{-1} \epsilonb \ge - \sqrt{2 \xb^\top \Ub (\Ab')^{-2} \Ub^\top \xb \log (1 / \alpha)}
\]
Because $\Ub (\Ab')^{-1}\Ub^\top$ is positive definite with eigenvalues bounded by $1$, we have that $\xb^\top \Ub (\Ab')^{-2} \Ub^\top \xb \le \xb^\top \Ub (\Ab')^{-1} \Ub^\top \xb$.
We will now show that $\xb^\top \Ub (\Ab')^{-1} \Ub^\top \xb \ge \mu^2 / 2$ under $H^{PC}_1, H^S_1$.
Recall that by the dual norm (as derived in the proof of Prop.~\ref{prop:GFSS_def}), 
\begin{equation}\label{eq:dual_ellipse}
\xb^\top \Ub \Ab' \Ub^\top \xb = \sup_{\zb^\top \Ub (\Ab')^{-1} \Ub^\top \zb \le 1} (\zb^\top \xb)^2.
\end{equation}

{\em Case 1: $H_1^{PC}$.}  In this case,
\[
\frac{(\xb - \bar\xb)^\top}{\| \xb - \bar \xb\|} \Ub \Lambda \Ub^\top \frac{(\xb - \bar\xb)}{\|\xb - \bar \xb\|} \le \rho
\]
while $\| (\xb - \bar \xb) / \| \xb - \bar \xb \| \| = 1$. 
Thus,
\[
\frac{(\xb - \bar\xb)^\top}{\| \xb - \bar \xb\|} \Ub \Ab' \Ub^\top \frac{(\xb -
\bar\xb)}{\| \xb - \bar \xb\|} \le 2.
\]
So,
\[
\frac{(\xb - \bar\xb)^\top}{\sqrt 2 \| \xb - \bar \xb \|} \Ub \Ab' \Ub^\top \frac{(\xb - \bar\xb)}{\sqrt 2 \| \xb - \bar \xb \|} \le 1
\]
By substituting $z = (\xb - \bar \xb) / \sqrt 2 \|\xb - \bar \xb\|$ in \eqref{eq:dual_ellipse} we arrive at
\[
\Rightarrow \xb^\top \Ub (\Ab')^{-1} \Ub^\top \xb \ge \left(\frac{(\xb - \bar\xb)^\top}{\sqrt 2 \| \xb - \bar \xb \|} \xb \right)^2 = \| \xb - \bar \xb\|^2 / 2 \ge \mu^2 / 2,
\]
where the last inequality is due to the fact that $\xb \in \Xcal_{PC}(\mu,\rho)$.

{\em Case 2: $H_1^S$.}  Let $\xb \in \Xcal_S(\mu,C)$.  
In this case we will let $\Kb_C$ be the projection onto the span of $\one_C, \one_{\bar C}$ and orthogonal to $\one$.
So, \[
\Kb_C \xb = \frac{\one_C^\top \xb}{|C|} \one_C + \frac{\one_{\bar C}^\top \xb}{|\bar C|} \one_{\bar C}- \bar \xb.
\]
Let $\zb = \Kb_C \xb / \| \Kb_C \xb \|$ that $\zb^\top \xb = \| \Kb_C \xb \|$.
Let 
\[
\bar \xb_C = \frac{\one_C^\top \xb}{|C|} \one_C \quad {\rm and }\quad \bar \xb_{\bar C} = \frac{\one_{\bar C}^\top \xb}{|\bar C|} \one_{\bar C}.
\]
\[
\begin{aligned}
\bar \xb_C - \bar \xb = (\frac{1}{|C|} - \frac 1n) \one_C^\top \xb - \frac 1n \one_{\bar C}^\top \xb \\
= \frac {|\bar C|} n (\bar \xb_C - \bar \xb_{\bar C}).
\end{aligned}
\]
Similarly, $\bar \xb_{\bar C} - \bar \xb = \frac {|C|} n (\bar \xb_{\bar C} - \bar \xb_{C})$.
And so,
\[
\begin{aligned}
\zb^\top \xb = \|\Kb_C \xb\| = |C| \frac {|\bar C|^2} {n^2} (\bar \xb_C - \bar \xb_{\bar C})^2 + |\bar C| \frac {|C|^2} {n^2} (\bar \xb_{\bar C} - \bar \xb_{C})^2 \\
= \frac{|C||\bar C|}{n} (\bar \xb_{\bar C} - \bar \xb_{C})^2 \ge \mu^2.
\end{aligned}
\]
Now we can go through the same proof as the previous case substituting $\zb$ for $\xb - \bar \xb / \|\xb - \bar \xb\|$.

The error bound, \eqref{eq:alt_control} follows from these facts and the Lemma \ref{lem:chi_squared} applied to $\epsilonb^\top (\Ab')^{-1} \epsilonb - \tr (\Ab')^{-1}$.
\vspace{-.25cm}
\end{proof}
\vspace{-.1cm}

\begin{proof}[Proof of Corollary \ref{cor:BBT}]
The study of the spectra of trees really began in earnest with the work of \cite{fiedler1975eigenvectors}.
Notably, it became apparent that trees have eigenvalues with high multiplicities, particularly the eigenvalue $1$.
\cite{molitierno2000tight} gave a tight bound on the algebraic connectivity of balanced binary trees (BBT).
They found that for a BBT of depth $\ell$, the reciprocal of the smallest eigenvalue ($\lambda_2^{(\ell)}$) is 
\begin{equation}
\label{eqn:tree_eig_bound}
\begin{aligned}
\frac{1}{\lambda_2^{(\ell)}} \le 2^\ell - 2\ell + 2 - \frac{2^\ell - \sqrt{2} (2\ell -1 - 2^{\ell-1})}{2^\ell - 1 - \sqrt 2 (2^{\ell - 1} - 1)} \\
+ (3 - 2 \sqrt 2 \cos (\frac{\pi}{2\ell - 1}))^{-1} \\
\le 2^\ell + 105 I\{ \ell < 4 \}
\end{aligned}
\end{equation}
\cite{rojo2002spectrum} gave a more exact characterization of the spectrum of a balanced binary tree, providing a decomposition of the Laplacian's characteristic polynomial.
Specifically, the characteristic polynomial of $\Delta$ is given by
\begin{equation}
\label{eqn:tree_char_poly}
\begin{aligned}
\det (\lambda \Ib - \Delta) = p_1^{2^{\ell - 2}}(\lambda) p_2^{2^{\ell - 3}}(\lambda) \\
... p_{\ell - 3}^{2^2}(\lambda) p_{\ell - 2}^2(\lambda) p_{\ell - 1}(\lambda) s_\ell(\lambda)  
\end{aligned}
\end{equation}
where $s_\ell(\lambda)$ is a polynomial of degree $\ell$ and $p_i(\lambda)$ are polynomials of degree $i$ with the smallest root satisfying the bound in \eqref{eqn:tree_eig_bound} with $\ell$ replaced with $i$.
In \cite{rojo2005spectra}, they extended this work to more general balanced trees.

By \eqref{eqn:tree_char_poly} we know that at most $\ell + (\ell - 1) + (\ell - 2)2 + ... + (\ell - j)2^{j - 1} \le \ell 2^j$ eigenvalues have reciprocals larger than $2^{\ell - j} + 105 I\{ j < 4 \}$.
Let $k = \max \{ \lceil \frac{\ell}{c} 2^{\ell (1 - \alpha)} \rceil, 2^3\}$, then we have ensured that at most $k$ eigenvalues are smaller than $\rho$.
For $n$ large enough
\[
\begin{aligned}
\sum_{i > 1} \min\{1, \rho^2 \lambda_i^{-2}\} \le k + \rho^2 \sum_{j > \log k}^{\ell} \ell 2^j 2^{2(\ell - j)} \\
\le k + \frac \ell k n^2 \rho^2 = O(n^{1-\alpha} \log n)   
\end{aligned}
\]
\vspace{-.25cm}
\end{proof}
\vspace{-.1cm}

\begin{proof}[Proof of Cor.~\ref{cor:torus}]
By a simple Fourier analysis (see \cite{sharpnack2010identifying}), we know that the Laplacian eigenvalues are $2 (2 - \cos (2 \pi i_1/\ell) - \cos (2 \pi i_2 / \ell))$ for all $i_1,i_2 \in [\ell]$.
Let us denote the $\ell^2$ eigenvalues as $\lambda_{(i_1,i_2)}$ for $i_1, i_2 \in [\ell]$.
Notice that for $i \in [\ell]$, $|\{(i_1,i_2) : i_1 \vee i_2 = i \}| \le 2 i$.
For simplicity let $\ell$ be even.
We know that if $i_1 \vee i_2 \le \ell/2$ then $\lambda_{(i_1,i_2)} = 2 - \cos(2 \pi i_1/ \ell) - \cos(2 \pi i_2 / \ell) \ge 1 - \cos(2 \pi (i_1 \vee i_2) / \ell)$ . Let $k \ll \ell$ which we will specify later.  Thus,
\[
\begin{aligned}
&\sum_{(i_1,i_2) \ne (1,1) \in [\ell]^2} 1 \wedge \frac{\rho^2}{\lambda_{(i_1,i_2)}^2} \\
&\le 2 \sum_{i \in [\ell/2]} 2 i \left( 1 \wedge \frac{\rho^2}{ (1 - \cos(2 \pi i / \ell))^2 } \right) \\
&\le 4 \sum_{i = 1}^k i + \rho^2 \frac{\ell^2}{2} \frac 2\ell \sum_{k < i \le \ell/2} 2 \frac{i / \ell}{ (1 - \cos(2 \pi i / \ell))^2 }\\
&\le 4 k^2 + \rho^2 \frac{\ell^2}{2} \int_{k/\ell}^{1/2} \frac{x dx}{(1 - \cos(2 \pi x))^2}\\
& = 4 k^2 + \rho^2 \frac{\ell^2}{2} \left(\frac{1}{4\pi^4 }\frac{\ell^3}{k^3} + O\left(\frac \ell k\right)\right)
\end{aligned}
\]
The above followed by the Taylor expansion about $0$ of the integral.
Let us choose $k$ to such that $k \approx \rho^{2/5} \ell$.
The inequalities above hold regardless of the choice of $k$, as long $k \ll \ell$, so we have the freedom to tune it to our liking.
Plugging this in we obtain,
\[
\sum_{(i_1,i_2) \ne (1,1) \in [\ell]^2} 1 \wedge \frac{\rho^2}{\lambda_{(i_1,i_2)}^2} = O(\rho^{4/5} \ell^{2}) = O(p^{3/5  + 2\beta/5})
\]
\vspace{-.25cm}
\end{proof}
\vspace{-.1cm}

\begin{proof}[Proof of Corollary \ref{cor:kron}]
The Kronecker product of two matrices $\Ab, \Bb \in \RR^{n \times n}$ is defined as $\Ab \otimes \Bb \in \RR^{(n \times n) \times (n \times n)}$ such that $(\Ab \otimes \Bb)_{(i_1,i_2),(j_1,j_2)} = A_{i_1,j_1} B_{i_2,j_2}$.
Some matrix algebra shows that if $H_1$ and $H_2$ are graphs on $p$ vertices with Laplacians $\Delta_1, \Delta_2$ then the Laplacian of their Kronecker product, $H_1 \otimes H_2$, is given by $\Delta = \Delta_1 \otimes \Ib_p + \Ib_p \otimes \Delta_2$ (\cite{merris1998laplacian}).
Hence, if $\vb_1, \vb_2 \in \RR^p$ are eigenvectors, viz.~$\Delta_1 \vb_1 = \lambda_1 \vb_1$ and $\Delta_2 \vb_2 = \lambda_2 \vb_2$, then $\Delta (\vb_1 \otimes \vb_2) = (\lambda_1 + \lambda_2) \vb_1 \otimes \vb_2$, where $\vb_1 \otimes \vb_2$ is the usual tensor product.
This completely characterizes the spectrum of Kronecker products of graphs.

We should argue the choice of $\rho \asymp p^{2k - \ell - 1}$, by showing that it is the results of cuts at level $k$.
We say that an edge $e = ((i_1,...,i_\ell),(j_1,...,j_\ell))$ has scale $k$ if $i_k \ne j_k$.
Furthermore, a cut has scale $k$ if each of its constituent edges has scale at least $k$.
Each edge at scale $k$ has weight $p^{k - \ell}$ and there are $p^{\ell-1}$ such edges, so cuts at scale $k$ have total edge weight bounded by 
\[
p^{\ell - 1} \sum_{i = 1}^k p^{i - \ell} = p^{k - 1} \frac{p - \frac{1}{p^{k-1}}}{p - 1} \le \frac{p^k}{p - 1}
\]
Cuts at scale $k$ leave components of size $p^{\ell - k}$ intact, meaning that $\rho \propto p^{2k - \ell - 1}$ for large enough $p$. 

We now control the spectrum of the Kronecker graph.
Let the eigenvalues of the base graph $H$ be $\{\nu_j \}_{j=1}^p$ in increasing order.
The eigenvalues of $G$ are precisely the sums
\[
\lambda_i = \frac{1}{p^{\ell-1}} \nu_{i_1} + \frac{1}{p^{\ell-2}} \nu_{i_2} + ... + \frac{1}{p} \nu_{i_{\ell-1}} + \nu_{i_\ell}
\]
for $i = (i_j)_{j = 1}^\ell \subseteq V$.
The eigenvalue distribution $\{ \lambda_i \}$ stochastically bounds 
\[
\lambda_i \ge \sum_{j = 1}^\ell \frac{1}{p^{\ell-j}} \nu_2 I\{\nu_{i_j} \ne 0\} \ge \frac{\nu_2}{p^{Z(i)}}
\]
where $Z(i) = \min \{j : \nu_{i_{\ell - j}} \ne 0\}$.
Notice that if $i$ is chosen uniformly at random then $Z(i)$ has a geometric distribution with probability of success $(p - 1)/p$.
Hence,
\[
\begin{aligned}
&\frac{1}{p^\ell}\sum_{i \in V^\ell} \min\{1 , \frac{\rho^2}{\lambda^2_i}\} \le \EE_Z \min\{1 , \frac{\rho^2 p^{2 Z}}{\nu_2^2}\} \\
& \le \PP_Z \{Z \ge 2 k - \ell - 1 + \log_p \nu_2\} \\
&+ \frac{1}{\nu_2^2}\sum_{z = 1}^{\lfloor \ell + 1 - 2k + \log_p \nu_2 \rfloor} p^{2(2 k - \ell - 1 + z)} \PP_Z \{Z = z\}\\
& \le p^{ 2k - \ell - 1 + \log_p \nu_2} \\
&+ \frac{1}{\nu_2^2} \sum_{z = 1}^{\lfloor \ell + 1 - 2k + \log_p \nu_2 \rfloor} p^{2(z + 2k - \ell - 1)} \frac 1{p^z} \frac{p - 1}{p} \\
& = O((\nu_2 + \nu_2^{-1}) p^{2k - \ell - 1}) = O(p^{2k - l} \textrm{diam}(H))
\end{aligned}
\]
where $\textrm{diam}(H)$ is the diameter of the base graph $H$.  Hence, 
\[
\sum_{i \in V^\ell} \min\{1 , \frac{\rho^2}{\lambda^2_i}\} = O(n^{2k/\ell} \textrm{diam}(H))
\]
\vspace{-.25cm}
\end{proof}
\vspace{-.25cm}

\section{The LR Statistic}
Below we will provide the details for the derivation of the LR
statistic \eqref{eq:LR} for testing the null hypothesis that $\xb = \bar \xb$
versus the alternative hypothesis 
\[
    \xb = \alpha \one + \delta \one_C, : \alpha, \delta \in
    \RR, \delta \neq 0, 
\]
for one given non-empty $C \subset V$. The unknown parameter $\alpha$ is  a nuisance
parameter.

To eliminate the dependence on $\alpha$ and simplify the
problem we will resort to invariant testing theory \cite{lehmann2005testing}. In fact, the testing problem remains invariant under the action of
the group of translations,  i.e.  additions of constant vectors, of
the mean of $\yb$. To take
advantage of such invariance we proceed as follows.  Let $\Bb$ be a $(p-1) \times p$ whose rows form an orthonormal basis for
$\mathcal{R}^\bot(\one)$, the linear subspace of $\mathbb{R}^P$ orthogonal to the subspace of vectors in $\mathbb{R}^p$ with constant
entries (the matrix $\Ub^\top$ as defined in the proof of Prop.~\ref{prop:GFSS_def} would suffice). Then, a maximal invariant with respect with respect to such a group is the $(p-1)$-dimensional random vector
\[
\zb := \Bb \yb = \Bb \one_C \delta + \Bb \epsilonb.
    \]
Since $\Bb \Bb^\top =
    \Ib_{p-1}$,  $\zb$ has a $N_{p-1}(\Bb \one_C \delta, \sigma^2 \Ib_{p-1})$
distribution, which no longer depends on the nuisance parameter $\alpha$.
Our hypothesis testing problem is then equivalent to the problem of testing $H_0
\colon \mathbb{E}[\zb] = 0$ versus the alternative $H_1^C \colon \mathbb{E}[\zb] =
\delta \Bb \one_C$. It is also worth pointing out that, as our calculations below show,
the choice of the orthonormal basis of $\mathcal{R}^\bot(\one)$ comprising the
rows of the matrix $\Bb$ does not matter in the construction of the
optimal test.


The LR statistic is
\[
    \frac{\sup_{\delta \in \mathbb{R}} \exp \left\{ -\frac{1}{2 \sigma^2} \left\|\zb - \Bb \one_C \delta
\right\|^2  \right\} }{ \exp\left\{
    - \frac{1}{2 \sigma^2} \| \zb \|^2
\right\}}.
    \]
    Simple calculations yield that MLE of $\delta$ under the alternative is $\frac{\zb^\top \Bb\one_C}{\| \Bb
    \one_C \|^2}$. As a result, the LR becomes
    \[
	\exp\left\{ -\frac{1}{2 \sigma^2} \left[ \left\| \zb - \Bb\one_C \frac{\zb^\top
	\Bb\one_C}{\| \Bb\one_C \|^2} \right\|^2  - \| \zb \|^2 \right]
    \right\},
    \]
which is equal to 
\begin{equation}\label{eq:LR2}
    \exp\left\{ \frac{1}{2 \sigma^2} \frac{( \zb^\top\Bb\one_C)^2 }{\| \Bb\one_C\|^2 }
\right\}.
    \end{equation}
We now rewrite the previous display in a simpler form. We have 
\[
\zb^\top \Bb \one_C = \yb^\top \Bb \Bb^\top \one_C = \one_C^\top \Kb \yb = \sum_{i \in C} \tilde y_i,
\]
where $\Kb = \Ib_p - \frac{\one \one^\top}{p}$ is the orthogonal projector into the subspace of
$\mathbb{R}^p$ orthogonal to the linear subspace spanned by the constant vectors.  
Next, since $\Kb$ is idempotent, we have 
\begin{eqnarray*}
\| \Bb \one_C \|^2 & = & \one_C^\top \Kb \one_C = \|  \Kb \one_C\|^2\\
& = & \sum_{i \in C} \left( 1 - \frac{|C|}{p} \right)^2 + \sum_{i \in \bar C} \left( - \frac{|\bar C|}{p} \right)^2 \\
& = \frac{|C||\bar C|}{p},\\
\end{eqnarray*}
where in last
equality we used the fact that $|C| + |\bar C| = p$. Plugging into
\eqref{eq:LR2},
we arrive at the expression for the log-likelihood ratio in \eqref{eq:LR}.

\end{document}